\documentclass[a4paper,10pt,psamsfonts]{amsart}

\usepackage[utf8]{inputenc} % direct use of umlauts
\usepackage{amsmath}
\usepackage{amstext}
\usepackage{amsfonts}
\usepackage{amsthm}
\usepackage{amssymb}
\usepackage[bf,SL,BF]{subfigure}
\usepackage{graphicx}
\usepackage{caption}
\usepackage{listings}
\usepackage{hyperref}
\hypersetup{colorlinks}
\usepackage{color}

\usepackage{changes}

\newcommand{\R}{{\mathbb{R}}}

\newcommand{\E}{{\mathbb{E}}}
\newcommand{\N}{{\mathbb{N}}}

\newcommand{\F}{{\mathcal{F}}} % for the filtration
\renewcommand{\P}{{\mathbb{P}}} % for the probability measure

\newcommand{\diff}[1]{\,\mathrm{d}#1}

\newcommand{\ee}{\mathrm{e}}

\newcommand{\triple}{{\vert\kern-0.25ex\vert\kern-0.25ex\vert}}

\theoremstyle{plain}

\newtheorem{definition}{Definition}[section]
\newtheorem{theorem}[definition]{Theorem}
\newtheorem{lemma}[definition]{Lemma}
\newtheorem{corollary}[definition]{Corollary}

\newtheorem{assumption}[definition]{Assumption}

\theoremstyle{definition}
\newtheorem{remark}[definition]{Remark}

\begin{document}

\title[A drift-randomized Milstein method for SDEs]
{A randomized Milstein method for\\
stochastic differential equations with\\
non-differentiable drift coefficients}

\author[R.~Kruse]{Raphael Kruse}
\address{Raphael Kruse\\
Technische Universit\"at Berlin\\
Institut f\"ur Mathematik, Secr. MA 5-3\\
Stra\ss e des 17.~Juni 136\\
DE-10623 Berlin\\
Germany}
\email{kruse@math.tu-berlin.de}

\author[Y.~Wu]{Yue Wu}
\address{Yue Wu\\
Technische Universit\"at Berlin\\
Institut f\"ur Mathematik, Secr. MA 5-3\\
Stra\ss e des 17.~Juni 136\\
DE-10623 Berlin\\
Germany}
\email{wu@math.tu-berlin.de}
\keywords{Milstein method, stochastic differential equations, strong
convergence, Monte Carlo methods, randomized quadrature rules} 
\subjclass[2010]{65C30, 65C05, 65L20, 60H10} 

\begin{abstract}
  In this paper a drift-randomized Milstein method is introduced 
  for the numerical solution of non-autonomous stochastic differential
  equations with non-differentiable drift coefficient functions. 
  Compared to standard Milstein-type methods
  we obtain higher order convergence rates in the $L^p(\Omega)$ and 
  almost sure sense. An important ingredient in the error analysis
  are randomized quadrature rules for H\"older continuous stochastic 
  processes.  By this we avoid the use of standard arguments based on the
  It\=o-Taylor expansion which are typically applied in
  error estimates of the classical Milstein method but
  require additional smoothness of the drift and diffusion coefficient
  functions. We also discuss the optimality of our convergence rates. Finally,
  the question of implementation is addressed in a numerical experiment.
\end{abstract}

\maketitle
%\tableofcontents

\section{Introduction}
\label{sec:intro}

For many decades the numerical solution of stochastic differential equations
(SDEs) has been a very active research area in the intersection of probability
and numerical analysis. A wide range of applications, for instance, in the
engineering and physical sciences as well as in computational finance is still
spurring the demand for the development of more efficient algorithms and their
theoretical justification. In particular, 
the current focus lies on the approximation of 
SDEs which cannot be treated by standard methods found in the pioneering
books of P.~E.~Kloeden and E.~Platen \cite{kloeden1999}, or G.~N.~Milstein and
M.~V.~Tretyakov \cite{milstein1995, milstein2004}. 

Due to the presence of an irregular stochastic forcing term, solutions to SDEs
are typically non-smooth. This makes it notoriously difficult to construct
higher order numerical approximations. The first successful attempt to
construct a first order numerical algorithm for the approximation
of an SDE with multiplicative
noise led to the well-known Milstein method \cite{milstein1974, milstein1975}.
Its derivation is based on the It\=o-Taylor formula and it can be generalized
to construct approximations of, in principle, arbitrary high order provided the
coefficient functions are sufficiently smooth. We again refer to the monographs
\cite{kloeden1999, milstein1995, milstein2004}. 

Unfortunately, the standard smoothness and growth requirements are often
not fulfilled in applications. For instance, already 
in the case of super-linearly growing coefficient functions, the standard
Euler-Maruyama and Milstein methods are known to be divergent in the strong and
weak sense, see \cite{hutzenthaler2011}. It is therefore necessary to apply
these methods only with caution if the SDE in question
does not fit into the framework of 
\cite{kloeden1999, milstein1995, milstein2004}.
In this paper we focus on the numerical solution of 
non-autonomous SDEs whose drift coefficient functions are not necessarily
differentiable. We will show that a higher order approximation of the
exact solution that outperforms the Euler-Maruyama method 
can still be obtained in this case by using suitable Monte Carlo randomization
techniques. 

To be more precise, 
let $T \in (0,\infty)$ and $(\Omega_W, \F^W, (\F^W_t)_{t
\in [0,T]}, \P_W)$ be a filtered probability space satisfying the usual
conditions. For $d,m \in \N$ let $W \colon [0,T] \times \Omega_W \to \R^m$ be a
standard $(\F^W_t)_{t \in [0,T]}$-Wiener process. Moreover, let $X \colon 
[0,T]\times\Omega_W \to \R^d$ be an $(\F^W_t)_{t \in [0,T]}$-adapted stochastic
process that is a solution to the It\=o-type stochastic differential equation  
\begin{align}
  \label{eq:SODE}
  \begin{split}
  \begin{cases}
    \diff{X(t)} &= f(t, X(t)) \diff{t} + \sum_{r=1}^m g^r\big(t,X(t)\big)
    \diff{W^r(t)}, \quad t \in [0,T],\\ 
    X(0) &= X_0,
  \end{cases}
  \end{split}
\end{align}
where $X_0  \in L^{2p}(\Omega_W, \F^W_0,\P_W; \R^d)$ for some $p \in
[2,\infty)$ denotes the initial value. The drift coefficient function $f \colon
[0,T] \times \R^d \to \R^d$ and the diffusion coefficient functions $g^r \colon
[0,T]\times \R^d \to \R^{d}$ for $r\in \{1,2,\ldots,m\}$ are assumed to satisfy
certain Lipschitz and linear growth conditions. For a complete statement of all
conditions on $f$ and $g^r$ we refer to Section~\ref{sec:str_error}.

If the drift function $f$ is only $\gamma$-H\"older
continuous, $\gamma \in (0,1]$,  with respect to the time variable and
Lipschitz continuous with respect to the state variable (see
Assumption~\ref{as:f}), then it is well-known that in the deterministic
case ($g^r \equiv 0$ for all $r \in \{1,\ldots,m\}$)
the order of convergence of the standard Euler method can, in general, not
exceed $\gamma$. This is even true for any deterministic algorithm that only
uses finitely many point evaluations of the drift $f$, see \cite{heinrich2008,
kacewicz1987}. 

One possibility to increase the order of convergence in such a case consists
of a suitable combination of the one-step method with certain
Monte-Carlo techniques. For deterministic differential equations this has been
studied, for example, in \cite{daun2011, heinrich2008, jentzen2009,
kacewicz2006, kruse2017, stengle1990, stengle1995}. In particular, 
in \cite{daun2011, heinrich2008, kruse2017} certain randomized Euler and
Runge-Kutta methods are introduced which converge with order $\gamma +
\frac{1}{2}$ under the same smoothness assumptions on $f$ as above. In fact,
these convergence rates are shown to be optimal within the class of all
randomized algorithms, see \cite{heinrich2008}.

The purpose of this paper is to combine these randomization techniques with the
classical Milstein scheme in order to obtain a higher order approximation
method in the case of a non-differentiable drift coefficient function $f$. 
For the introduction of the resulting \emph{drift-randomized Milstein
method} let $\pi_h$ be a not necessarily equidistant temporal grid of the form 
\begin{align}
  \label{eq:grid}
  \pi_h:=\{t_j\; :\; j=0,1,\ldots,N_h, 
  \; 0=t_0<t_1<\ldots < t_{N_h-1}< t_{N_h}=T\},
\end{align}
where $N_h \in \N$ and $h_j:=t_j-t_{j-1}$ is the width of the $j$-th step.
Given a temporal grid $\pi_h$ we denote the associated vector of all step sizes
by 
\begin{equation}
  \label{eq:stepsizes}
  h:=(h_j)_{j=1}^{N_h} \in \R^{N_h} \quad \text{ with } t_n = \sum_{j=1}^{n}
  h_j.
\end{equation}
The maximum step size in $\pi_h$ is then denoted by
\begin{align*}
  |h| := \max_{j\in \{1,\ldots,N_h\}} h_j.
\end{align*}
Further, let $(\tau_j)_{j \in \N}$ be an i.i.d.~family of
$\mathcal{U}(0,1)$-distributed random variables on an additional filtered
probability space $(\Omega_{\tau}, \F^{\tau}, (\F^{\tau}_j)_{j\in \N},
\P_{\tau})$, where $\F_j^{\tau}$ is the $\sigma$-algebra generated by
$\{\tau_1,\ldots, \tau_j\}$. The random variables $(\tau_j)_{j \in \N}$
represent the artificially added 
random input for the new method, which we assume
to be independent of the randomness already 
present in SDE \eqref{eq:SODE}.

The resulting numerical method will then yield a discrete-time stochastic
process defined on the product probability space 
\begin{align}
  \label{eq:Omega}
  (\Omega, \F, \P) := (\Omega_{W}\times\Omega_{\tau}, \F^W\otimes \F^{\tau},
  \P_W \otimes \P_{\tau}).
\end{align}
Moreover, for each temporal grid $\pi_h$ a 
discrete-time filtration $(\F^h_n)_{n \in \{1,\ldots,N_h\}}$
on $(\Omega, \F, \P)$ is given by 
\begin{align}
  \label{eq:filtration}
  \F^h_n := \F^W_{t_n} \otimes \F^{\tau}_n, \quad \text{ for } n \in
  \{0,1,\ldots,N_h\}.
\end{align}
Finally, for the formulation of the drift-randomized Milstein method,
we also recall the following standard notation for the stochastic increments
and iterated stochastic integrals (c.f.\cite{kloeden1999, milstein1995,
milstein2004}): For $s, t \in [0,T]$ with $s < t$ set 
\begin{align}
  \label{eq:1stinterated}
  I_{(r)}^{s,t} &:= \int_s^t\diff{W^r(u)}, \ \ \ \mbox{for \ } 
  r\in\{1,2,\ldots,m\},\\
  \label{eq:2ndinterated}
  I_{(r_1,r_2)}^{s,t}&:=\int_s^t\int_s^{u_1}
  \diff{W^{r_1}(u_2)}\diff{W^{r_2}(u_1)},
  \ \ \  \mbox{for\ }r_1, r_2\in\{1,2,\ldots,m\}.
\end{align}
We further introduce the mapping $g^{r_1,r_2} \colon
[0,T] \times \R^d \to \R^d$ given by
\begin{equation}
  \label{eq:diffg}
  g^{r_1,r_2}(t,x):=\frac{\partial g^{r_1}}{\partial x}(t,x)g^{r_2}(t,x),
\end{equation}
for all $r_1, r_2\in\{1,2,\ldots,m\}$, $t \in [0,T]$, $x \in \R^d$.
Then, the \emph{drift-randomized Milstein method} on the grid $\pi_h$ is given
by the split-step recursion 
\begin{align}
  \label{eq:RandM}
  \begin{split}
    X_h^{j, \tau} = &X_h^{j-1} + \tau_j h_j f\big( t_{j-1}, X_h^{j-1} \big)
    + \sum_{r=1}^m g^r\big(t_{j-1},X_h^{j-1} \big)I^{t_{j-1},t_{j-1}+ \tau_j
    h_j }_{(r)}, \\ 
    X_h^j =& X_h^{j-1} + h_jf( t_{j-1} + \tau_j h_j,  X_h^{j, \tau} )
    +\sum_{r=1}^m  
    g^r(t_{j-1},X_h^{j-1})I^{t_{j-1},t_j}_{(r)}\\
    & +\sum_{r_1,r_2=1}^{m} g^{r_1,r_2}(t_{j-1}, X_h^{j-1}) 
    I_{(r_2,r_1)}^{t_{j-1},t_j},   
  \end{split}
\end{align}
for all $j \in \{1, \ldots, N_h\}$, and the initial value $X_h^0 = X_0$.

The main result of this paper then shows that this
method converges to the exact solution with respect to the norm in
$L^p(\Omega)$, $p \in [2,\infty)$. More precisely, Theorem~\ref{th:main2}
states that under Assumptions~\ref{as:X0} to \ref{as:g} there exists $C \in
(0,\infty)$ 
independent of the temporal grid $\pi_h$ such that
\begin{align*}
  \big\| \max_{n \in \{0,1, \ldots, N_h\} } | X_h^n - X(t_n)|
  \big\|_{L^p(\Omega)} \le C |h|^{\min(\frac{1}{2} + \gamma, 1)},
\end{align*}
where $\gamma \in (0,1]$ denotes as above the temporal 
H\"older regularity of the drift coefficient function. It turns out that this
convergence rate is optimal under these conditions on $f$ as we will discuss
in more detail in Section~\ref{sec:str_error}. In addition, it is
a simple consequence of Theorem~\ref{th:main2} that the drift-randomized
Milstein method is then also convergent in a pathwise sense,
see Corollary~\ref{cor:pathwise}.

In Section~\ref{sec:examples} we will also illustrate that the randomized
Milstein method is easily implemented for a scalar noise.  
For a multi-dimensional Wiener process
the joint simulation of the iterated stochastic integrals
\eqref{eq:2ndinterated} is, in general, very costly. Since this issue also
applies to the classical Milstein it is, however, not further addressed in this
paper. Instead we refer to the discussion in \cite[Chap.~5]{kloeden1999}.
Further approximation methods for the simulation of iterated stochastic
integrals are found, for instance, in \cite{gaines1994, ryden2001,
wiktorsson2001}. Moreover, it is worth mentioning that, besides the case of
commutative noise (see \cite[Chap.~10.3]{kloeden1999}), the simulation of the
iterated stochastic integrals \eqref{eq:2ndinterated} can also be
avoided if the Milstein method is combined with 
an antithetic multilevel Monte Carlo algorithm, see \cite{giles2014}.

Before we give an outline of the remainder of this paper, let us briefly
mention that drift-randomized one-step methods for the numerical solution of
SDEs have also been studied by P.~Przyby{\l}owicz and P.~Morkisz
\cite{morkisz2017, przybylowicz2015a, przybylowicz2015b, przybylowicz2014}.
Here the focus lies on randomized Euler-Maruyama type methods applied to SDEs,
whose drift-coefficient functions are of Carath\'eodory-type. In particular,
the authors derive optimal and minimal error estimates in the case of drift
coefficient functions, that are discontinuous with respect to the temporal
argument $t$.  

In the following sections we will first focus on the error analysis of the
drift-randomized Milstein method. To this end we fix further
notation and recall some useful results from stochastic analysis in
Section~\ref{sec:notation}. In Section~\ref{sec:str_error} we then formulate
the main result on the convergence of the drift-randomized Milstein method in
the $L^p(\Omega)$ and almost sure sense. In addition, this section also
includes a complete list of all
imposed conditions on the drift and diffusion coefficient functions and some
properties of the exact solution to \eqref{eq:SODE}. For the proof of our main
result stated in 
Theorem~\ref{th:main2} we then employ a framework developed in \cite{beyn2010}.
For this we first prove in Section~\ref{sec:bistab} that the method
\eqref{eq:RandM} is \emph{stochastically bistable}. The second ingredient in
the error analysis is then to show that the method is also \emph{consistent}.
This
will be done in Section~\ref{sec:consistency}. Our proof of consistency is
based on some error estimates for randomized quadrature rules applied to
stochastic processes. This result of possibly independent interest generalizes
error estimates for Monte Carlo integration from \cite{haber1966, haber1967} and
is presented in Section~\ref{sec:quadrature}. Finally, in
Section~\ref{sec:examples} we illustrate the practicability of the
drift-randomized Milstein method through a numerical experiment.

%%%%%%%%%%%%%%%%%%%%%%%%%%%%%%
% Preliminaries and notations
%%%%%%%%%%%%%%%%%%%%%%%%%%%%%%

\section{Notation and preliminaries}
\label{sec:notation}

In this section we explain the notation that is used throughout
this paper. In addition, we also collect a few standard results from stochastic
analysis, which are needed in later sections.

By $\N$ we denote the set of all positive integers, while $\N_0 := \N \cup
\{0\}$. As usual the set $\R$ consists of all real numbers. By $| \cdot |$ we
denote the Euclidean norm on the Euclidean space $\R^d$ for any $d \in \N$. In
particular, if $d = 1$ then $| \cdot |$ coincides with taking
the absolute value. Moreover, the norm $| \cdot |_{\mathcal{L}(\R^d)}$ denotes
the standard matrix norm on $\R^{d \times d}$ induced by the Euclidean
norm.

We will also frequently encounter normed function spaces. First, for an arbitrary
Banach space $(E, \|\cdot\|_E)$ we denote by $\mathcal{C}^\gamma([0,T] ; E)$
with $T \in (0,\infty)$ and $\gamma \in (0,1]$ the space of all
$\gamma$-H\"older continuous $E$-valued 
mappings $v \colon [0,T] \to E$ with norm
\begin{align*}
  \| v \|_{\mathcal{C}^\gamma([0,T];E)} = \sup_{t \in [0,T]} \| v(t) \|_E + 
  \sup_{\substack{t,s \in [0,T]\\ t \neq s}} 
  \frac{\| v(t) - v(s)\|_E}{|t-s|^\gamma}. 
\end{align*}
For a given measure space $(X, \mathcal{A}, \mu)$ the set
$L^p(X; E) := L^p( X, \mathcal{A}, \mu; E)$, $p \in [1,\infty)$, consists of all
(equivalence classes of) Bochner measurable functions $v \colon X \to E$ with
\begin{align*}
  \| v \|_{L^p(X;E)} := \Big( \int_X \|v(x)\|_E^p \diff{\mu(x)}
  \Big)^{\frac{1}{p}} < \infty.
\end{align*}
If $(E, \| \cdot\|_E ) = (\R, |\cdot|)$ we use the abbreviation $L^p(X) :=
L^p(X;\R)$. If $(X, \mathcal{A}, \mu) = (\Omega, \F, \P)$ is a probability
space, we usually write the integral with respect to the probability measure
$\P$ as
\begin{align*}
  \E[ Z ] := \int_\Omega Z(\omega) \diff{\P(\omega)}, \quad Z \in
  L^p(\Omega;E).
\end{align*}
In the case of the product probability space $(\Omega, \F, \P)$ introduced in
\eqref{eq:Omega} an application of Fubini's theorem shows that
\begin{align*}
  \E[Z] = \E_W[\E_{\tau}[ Z]] = \E_{\tau}[\E_{W}[Z]], \quad Z \in
  L^p(\Omega;E),
\end{align*}
where $\E_W$ is the expectation with respect to $\P_W$ and $\E_{\tau}$ with 
respect to $\P_{\tau}$. Finally, $\mathcal{U}(0,1)$ denotes the uniform
distribution on the interval $(0,1)$. 

An important tool is the following discrete-time version of the
Burkholder-Davis-Gundy inequality from \cite{burkholder1966}. 

\begin{theorem}
  \label{th:discreteBDG}
  For each $p \in (1,\infty)$ there exist positive constants $c_p$ and $C_p$
  such that for every discrete-time martingale $(Y^n)_{n \in \N_0}$ and
  for every $n \in \N_0$ we have 
  $$c_p \big\| [Y]_{n}^{\frac{1}{2}} \big\|_{L^p(\Omega)}
  \leq \big\| \max_{j\in \{0,\ldots,n\} } |Y^{j}| \big\|_{L^p(\Omega)} 
  \leq C_p \big\| [Y]_{n}^{\frac{1}{2}} \big\|_{L^p(\Omega)},$$
  where $[Y]_n = |Y^{0}|^2 + \sum_{k=1}^{n} |Y^{k}-Y^{k-1}|^2$
  is the \emph{quadratic variation} of $(Y^n)_{n \in \N_0}$.
\end{theorem}

The following theorem contains a useful estimate of stochastic It\=o-integrals
with respect to the $L^p(\Omega;\R^d)$-norm. For a proof we refer to
\cite[Section~1.7]{mao2008}.

\begin{theorem}
  \label{th:Lpstochint}
  Let $W \colon [0,T] \times \Omega_W \to \R$ be a standard $(\F_t^W)_{t \in
  [0,T]}$-Wiener process on $(\Omega_W, \F^W, \P^W)$. Let $Y \colon [0,T]
  \times \Omega_W \to \R^d$ be a stochastically integrable,
  $(\F_t^W)_{t \in [0,T]}$-adapted process with $Y \in L^p( [0,T] \times
  \Omega_W; \R^d)$ for some $p \in [2,\infty)$. Then, for all $t,s \in [0,T]$
  with $s < t$, it holds true that
  \begin{align*}
    \Big\| \int_s^t Y(u) \diff{W(u)} \Big\|_{L^p(\Omega_W;\R^d)}
    \le C_p (t-s)^{\frac{p-2}{2p}} \| Y \|_{L^p([s,t] \times \Omega_W;\R^d) }
  \end{align*} 
  with $C_p = (\frac{1}{2}p (p - 1) )^{\frac{1}{2}}$. 
\end{theorem}

The next inequality is a
useful tool to bound the error of a numerical approximation. For a proof we
refer, for instance, to \cite[Proposition 4.1]{emmrich1999}.
 
\begin{lemma}[Discrete Gronwall's inequality]
  \label{lem:Gronwall} Consider two nonnegative sequences $(u_n)_{n\in \N},
  (w_n)_{n\in \N} \subset \R$ which for some given $a \in [0,\infty)$ satisfy 
  $$u_n \leq a + \sum_{j=1}^{n-1} w_j u_j,\quad \text{ for all }  n \in \N.$$
  Then, for all $n \in \N$, it also holds true that
  $u_n \leq a \exp(\sum_{j=1}^{n-1} w_j ).$
\end{lemma}

%%%%%%%%%%%%%%%%%%%%%%%%%%%%%%%%%%%%%
% Assumptions and main result
%%%%%%%%%%%%%%%%%%%%%%%%%%%%%%%%%%%%%

\section{Assumptions and main results}
\label{sec:str_error}

In this section we present sufficient conditions for the convergence of
the drift-randomized Milstein method \eqref{eq:RandM}
with respect to the norm in $L^p(\Omega)$ for some $p \in [2,\infty)$.
After collecting a few important properties of the exact solution, we state and
discuss the main results of this paper, namely the convergence of the
method in the $L^p(\Omega)$-norm and in the almost sure sense.

\begin{assumption}
  \label{as:X0}
  There exists $p \in [2,\infty)$ such that the initial value 
  satisfies $X_0 \in L^{2p}(\Omega_W, \F_0^W, \P_W;\R^d)$. 
\end{assumption}

\begin{assumption}
  \label{as:f}
  The drift coefficient function $f \colon [0,T] \times \R^d \to
  \R^d$ is assumed to be continuous. Moreover, there exist $\gamma \in (0,1]$
  and $K_f \in (0, \infty)$ such that 
  \begin{align*}
    | f(t,x_1) - f(t, x_2) | &\le K_f | x_1 - x_2|,\\
    | f(t_1,x) - f(t_2, x) | &\le K_f( 1 + |x| ) |t_1 -
    t_2|^{\gamma},
  \end{align*}
  for all $t,t_1, t_2 \in [0,T]$, $x, x_1, x_2 \in \R^d$.  
\end{assumption}

For the formulation of Assumption~\ref{as:g} recall the definition of 
$g^{r_1,r_2}$ from \eqref{eq:diffg}.

\begin{assumption}
  \label{as:g}
  The diffusion coefficient functions $g^r \colon [0,T]\times\R^d \to 
  \R^{d}$, $r \in \{1,\ldots,m\}$, are assumed to be continuous. In addition,
  we assume that for every fixed $t \in [0,T]$ and $r \in
  \{1,\ldots,m\}$ the mapping $\R^d \ni x \mapsto
  g^r(t,x) \in \R^d$ is continuously differentiable. 
  Moreover, there exist $\gamma \in (0,1]$ and
  $K_g \in (0, \infty)$ with 
  \begin{align*}
    | g^{r}(t_1,x) - g^{r}(t_2,x) | &\le K_g(1+|x|) |t_1 -
    t_2|^{\min(\frac{1}{2} + \gamma,1)},\\
    \Big|\frac{\partial g^r}{\partial x}(t,x_1)-\frac{\partial g^r}{\partial
    x}(t,x_2)\Big|_{\mathcal{L}(\R^d)}&\leq K_g|x_1-x_2|,\\
    \Big| \frac{\partial g^r}{\partial x} (t,x) 
    \Big|_{\mathcal{L}(\R^d)} &\leq K_g,\\
    \big|g^{r_1,r_2}(t,x_1)-g^{r_1,r_2}(t,x_2)\big|&\leq K_g|x_1-x_2|
  \end{align*}
  for all $t_1, t_2 \in [0,T]$ and $x\in \R^d$ and
  $r,r_1,r_2\in\{1,2,\ldots,m\}$.
\end{assumption}

\begin{remark}
  (i) It directly follows from Assumption~\ref{as:f} that $f$ 
  satisfies a linear growth bound for all $t \in [0,T]$ and 
  $x \in \R^d$ of the form
  \begin{equation}
    \label{eq:lineargrowth}
    | f(t,x)| \le \tilde{K}_f \big(1 + |x| \big)
  \end{equation} 
  with $K_f \le \tilde{K}_f = \max(  K_f,   T^{\gamma} K_f + |f(0,0)|)$.

  (ii) The boundedness of $\frac{\partial g^r}{\partial x}$ 
  immediately implies that $g^r$, $r = 1,\ldots,m$, 
  is globally Lipschitz continuous. More precisely, for
  all $t \in [0,T]$ and $x_1, x_2 \in \R^d$ we have
  \begin{align}
    \label{eq:Lip_g}
    | g^{r}(t,x_1) - g^{r}(t,x_2) | &\le K_g |x_1 - x_2|. 
  \end{align}
  Together with the temporal H\"older continuity of $g^r$ this
  also implies a linear growth bound of the form
  \begin{align}
    \label{eq:lineargrowth_g}
    |g^r (t,x) | \le \tilde{K}_g \big( 1 + |x| \big)
  \end{align}
  with $K_g \le \tilde{K}_g = \max(  K_g , T^{\min(\frac{1}{2} + \gamma,1)} K_g 
  + \max_{r \in \{ 1,\ldots,m\}} |g^r(0,0)| )$. 
\end{remark}

Before moving to the main result, let us collect a few useful properties of the
exact solution $X$ to the SDE \eqref{eq:SODE}. A proof is found, e.g., in
\cite[Sect.~2.3, 2.4]{mao2008}. 

\begin{theorem}
  \label{th:Lp_est}
  Let Assumptions~\ref{as:X0} to \ref{as:g} be satisfied with $p \in
  [2,\infty)$. Then there exists an up to indistinguishability uniquely
  determined $(\F_t^W)_{t \in [0,T]}$-adapted 
  stochastic process  $X \colon [0,T] \times \Omega \to \R^d$
  satisfying \eqref{eq:SODE}. More precisely, for every $t \in [0,T]$
  it holds true that  
  \begin{align}
    \label{eq:exactint}
    X(t) = X_0 + \int_{0}^{t} f(s, X(s) ) \diff{s}
    + \sum_{r = 1}^m \int_{0}^{t} g^r(s, X(s) ) \diff{W^r(s)}
  \end{align}
  with probability one. Moreover,
  there exists $C \in (0,\infty)$ only depending on $\tilde{K}_f$,
  $\tilde{K}_g$, $p$, and $T$ such that
  \begin{equation}
     \label{eq:Lp1}
     \big\|  \sup_{t \in [0,T]} |X(t)|  \big\|_{L^{2p}(\Omega_W)}  
     \leq C \big( 1 + \big\| X_0 \big\|_{L^{2p}(\Omega_W;\R^d)} \big).
  \end{equation}
  In addition, for all $s, t \in [0,T]$ we have 
  \begin{equation}
    \label{eq:Lp2}
    \big\| X(t)-X(s) \big\|_{L^{2p}(\Omega_W;\R^d)}
    \le C \big( 1 + \big\| X_0 \big\|_{L^{2p}(\Omega_W;\R^d)} \big)
    |t-s|^{\frac{1}{2}}.
  \end{equation} 
  In particular, it holds $X 
  \in \mathcal{C}^{\frac{1}{2}}([0,T], L^{2p}(\Omega_W;\R^d))$
  with
  \begin{align*}
    \| X \|_{\mathcal{C}^{\frac{1}{2}}([0,T], L^{2p}(\Omega_W;\R^d))} 
    \le C \big( 1 + \big\| X_0 \big\|_{L^{2p}(\Omega_W;\R^d)} \big).
  \end{align*}
\end{theorem}

Let us now turn to the drift-randomized Milstein method \eqref{eq:RandM}.
In the following it is convenient to formally introduce the increment function
of the numerical method. For this let $\pi_h$ be an arbitrary temporal grid as
in \eqref{eq:grid}.
Then for each $j \in \{1,\ldots,N_h\}$ the increment function 
$\Phi^j_h \colon \R^d \times [0,1] \times \Omega_W \to \R^d$
of the $j$-th step is defined by
\begin{align}
  \label{eq:PhiM_expression}
  \begin{split}
    \Phi^j_h(y,\tau)&:= 
    h_j f ( t_{j-1} + \tau h_j, \Psi_h^j( y,\tau ) )
    + \sum_{r=1}^m g^r(t_{j-1},y) I^{t_{j-1},t_j}_{(r)} \\
    &\qquad + \sum_{r_1,r_2=1}^{m} g^{r_1,r_2}(t_{j-1},y)
    I_{(r_2,r_1)}^{t_{j-1},t_j},
  \end{split}
\end{align}
for all $y\in \R^d$ and $\tau\in [0,1]$, where
\begin{align}
  \label{eq:PsiM}
  \Psi_h^j(y,\tau) := y + \tau h_j f \big(
  t_{j-1}, y \big) + \sum_{r=1}^{m} g^r\big(t_{j-1}, y \big) 
  I^{t_{j-1}, t_{j-1}+\tau h_j}_{(r)}.
\end{align}
In terms of $\Phi_h$ we can then
rewrite the recursion defining the method \eqref{eq:RandM} by
\begin{align}
  \label{eq:transformM1}
  \begin{cases}
    X_h^{j}= X_h^{j-1}+\Phi^j_h(X_h^{j-1},\tau_j),& \quad j \in
    \{1,\ldots,N_h\}, \\ 
    X_h^0 = X_0.&
  \end{cases}
\end{align} 
The next lemma ensures that \eqref{eq:transformM1} indeed admits an adapted
sequence in $L^p(\Omega;\R^d)$.

\begin{lemma}
  \label{lem:Phi_prop}
  Let Assumptions~\ref{as:f} and \ref{as:g} be satisfied.
  Let $\pi_h$ be an arbitrary temporal grid and $j \in \{1, \ldots,N_h\}$. For
  every $Z \in L^p(\Omega, \F^h_{j-1},\P;\R^d)$, $p \in [2,\infty)$,
  it then holds true that
  \begin{align}
    \label{eq:phiwelldefined}
    \Phi_h^j(Z, \tau_j) \in L^p(\Omega, \F^h_{j},\P;\R^d).
  \end{align}
\end{lemma}

\begin{proof}
  From the continuity of $f$, $g^r$, and $g^{r_1, r_2}$ it follows 
  that $\Phi_h^j(Z, \tau_j) \colon \Omega \to \R^d$ is $\F^h_j$-measurable.
  Hence, it remains to prove the $L^p$ boundedness of $\Phi_h^j(Z,\tau_j)$.
  As in \eqref{eq:PhiM_expression} we split 
  $\Phi_h$ into three terms
  \begin{align*}
    \Phi^{j}_h(Z,\tau_j)=: \Pi^j_1+\Pi_2^j+\Pi_3^j.
  \end{align*}
  We give estimates for these terms separately. First, 
  for the estimate of $\Pi_2^j$ we have 
  \begin{align*}
    \|\Pi^j_2\|_{L^p(\Omega;\R^d)}& = \Big\|\sum_{r=1}^m
    g^r(t_{j-1},Z) I_{(r)}^{t_{j-1},t_j} 
    \Big\|_{L^p(\Omega;\R^d)}\\
    &\leq \sum_{r=1}^m \big\|g^r(t_{j-1},Z) 
    \big\|_{L^p(\Omega;\R^d)} \|I_{(r)}^{t_{j-1},t_j}\big\|_{L^p(\Omega)}\\
    &\leq m C_p \tilde{K}_g \big( 1+\|Z\|_{L^p(\Omega;\R^d)} \big)
    h_j^{\frac{1}{2}}<\infty,
  \end{align*}
  where the penultimate line is deduced from the triangle inequality and the
  independence of $Z$ and the increment of the Brownian motion
  $I_{(r)}^{t_{j-1},t_j}$. In addition,
  the last line follows from the linear growth \eqref{eq:lineargrowth_g} of
  $g$ and Theorem~\ref{th:Lpstochint} applied to the stochastic increment.

  The estimate of $\Pi_3^j:=\sum_{r_1,r_2=1}^m
  g^{r_1,r_2}(t_{j-1},Z) I_{(r_2,r_1)}^{t_{j-1},t_j}$ is obtained
  similarly by 
  \begin{align*}
    \|\Pi^j_3\|_{L^p(\Omega;\R^d)}
    &\leq \sum_{r_1,r_2=1}^m \big\|g^{r_1,r_2}(t_{j-1}, Z) 
    \big\|_{L^p(\Omega;\R^d)} 
    \big\|I_{(r_2,r_1)}^{t_{j-1},t_j} \big\|_{L^p(\Omega)}\\
    &\le m^2 C_p^2 K_g \tilde{K}_g \big(1+\|Z 
    \|_{L^p(\Omega;\R^d)}\big) h_j < \infty,
  \end{align*}
  where the last line is deduced from the linear growth of $g^{r_2}$ and
  the boundedness of the derivative of $g^{r_1}$. 
  In addition, by Theorem~\ref{th:Lpstochint}
  it holds true that 
  \begin{align}
    \label{eq5:est_stochincr}
    \big\| I^{t_{j-1}, t_{j}}_{(r_2,r_1)} \big\|_{L^p(\Omega)}
    &\le C_p^2 h_j
  \end{align}
  for all $r_1, r_2 \in \{1,\ldots,m\}$ with the same constant $C_p$ as above.
  
  It remains to show the $L^p$-estimate of $\Pi_1^j:=h_j f\big(t_{j-1}+\tau_j
  h_j,\Psi_{h}^j(Z,\tau_j) \big)$. The linear growth \eqref{eq:lineargrowth}
  of $f$ implies 
  \begin{align*}
    \|\Pi^j_3\|_{L^p(\Omega;\R^d)}
    \leq \tilde{K}_f \big( 1+\|\Psi_{h}^j(Z,\tau_j) \|_{L^p(\Omega;\R^d)}
    \big) h_j,
  \end{align*}
  where $\Psi_h$, defined in \eqref{eq:PsiM}, can be further estimated through
  the linear growth of both $f$ and $g^r$ as well as 
  Theorem~\ref{th:Lpstochint}: 
  \begin{align*}
    \| \Psi_{h}^j(Z,\tau_j) \|_{L^p(\Omega;\R^d)}
    &\le \|Z\|_{L^p(\Omega;\R^d)} + h_j \|f(t_{j-1},Z)\|_{L^p(\Omega;\R^d)}\\
    &\quad +\sum_{r=1}^m \big\|g^r(t_{j-1},Z ) \big\|_{L^p(\Omega;\R^d)}
    \|I_{(r)}^{t_{j-1},t_{j-1}+\tau_jh_j}\big\|_{L^p(\Omega)}\\
    &\leq \|Z\|_{L^p(\Omega;\R^d)}+\big(1+\|Z\|_{L^p(\Omega;\R^d)}\big)
    (\tilde{K}_f h_j + mC_p \tilde{K}_g h_j^{\frac{1}{2}})<\infty.
  \end{align*}  
  Here the estimate of the 
  increment $I_{(r)}^{t_{j-1},t_{j-1}+\tau_jh_j}$ comes
  from 
  \begin{align}
    \label{eq5:est_incr}
    \begin{split}
      \big\|I^{t_{j-1},t_{j-1}+\tau_j h_j}_{(r)}\big\|_{L^p(\Omega)}
      &= \big\| W^r(t_{j-1} +\tau_j h_j ) - W^r(t_{j-1} ) \big\|_{L^p(\Omega)}\\
      &= \big( \E_\tau\big[ \E_W [ | W^r(t_{j-1} +\tau_j h_j ) - W^r(t_{j-1} )
      |^p ] \big] \big)^{\frac{1}{p}}\\
      &\le \Big(\frac{p(p-1)}{2} \Big)^{\frac{1}{2}} h_j^{\frac{1}{2}} 
      \big( \E_\tau\big[ \tau_j^{\frac{p}{2}} \big] \big)^{\frac{1}{p}} 
      \le C_p h_j^{\frac{1}{2}}
    \end{split}
  \end{align}
  by an application of Theorem~\ref{th:Lpstochint}.
\end{proof}

\begin{definition}
  \label{def:strerr}
  We say that the numerical method \eqref{eq:RandM} \emph{converges with order
  $\beta \in (0,\infty)$ } to the exact solution $X$
  of \eqref{eq:SODE} in the $L^p(\Omega)$-norm  if there 
  exist $p \in [2, \infty)$, $C \in (0,\infty)$, $h_0 \in (0,T)$ such that
  for all temporal grids $\pi_h$ with $|h| \le h_0$ we have
  \begin{align*}
    \big\| \max_{n \in \{0,1, \ldots, N_h\} } | X_h^n - X(t_n)|
    \big\|_{L^p(\Omega)} \le C |h|^{\beta}.
  \end{align*}
  Here $(X_h^n)_{n\in \{0,1,\ldots, N_h\}} \subset L^p(\Omega;\R^d)$ is generated by
  \eqref{eq:RandM} on $\pi_h$.
\end{definition}

Next, we state our main result. The proof is deferred to the end of
Section~\ref{sec:consistency}.

\begin{theorem}
  \label{th:main2}
  Let Assumptions~\ref{as:X0} to \ref{as:g} be satisfied 
  with $p \in [2,\infty)$ and
  $\gamma \in (0,1]$. Then, the drift-randomized
  Milstein method \eqref{eq:RandM}   
  converges with order $\beta = \min(\frac{1}{2} + \gamma, 1)$ to the exact
  solution $X$ of \eqref{eq:SODE} in the $L^p(\Omega)$-norm.
\end{theorem}

We remark that the order of convergence $\min(\frac{1}{2} + \gamma, 1)$
is optimal in the following sense: First, recall that
the maximum order of convergence of the classical Milstein method is known to 
be $1$. This has been shown in \cite[Thm.~6.2]{kruse2012} by a
generalization of the well-known example of Clark and Cameron
\cite{clark1980}. Since that example does not contain a drift coefficient
function, the classical Milstein method and our randomized version
\eqref{eq:RandM} coincide in this case.
Therefore, the maximum order of convergence of 
\eqref{eq:RandM} cannot exceed $1$ as well.

Second, as already mentioned in Section~\ref{sec:intro},
in the ODE case
($g^r \equiv 0$ for all $r \in \{1,\ldots,m\}$) the maximum order of
convergence of randomized algorithms is known to 
be equal to $\frac{1}{2} + \gamma$ under Assumption~\ref{as:f}, see
\cite{heinrich2008}. In addition, it is shown in \cite{przybylowicz2014}
that the maximum order of convergence for the approximation of a stochastic
integral with $(\frac{1}{2} + \gamma)$-H\"older continuous integrand can also
not exceed $\frac{1}{2} + \gamma$. Therefore, there exists no (randomized)
algorithm, depending only on finitely many point evaluations of the
coefficients, that converges with a better rate than $\beta =
\min(\frac{1}{2} + \gamma, 1)$ for all $f$ and $g^r$ satisfying
Assumptions~\ref{as:f} and \ref{as:g}. 

We conclude this section with the following convergence result in the almost
sure sense. Its proof follows directly from Theorem~\ref{th:main2} and 
a modified version of \cite[Lemma~2.1]{kloeden2007} found in 
\cite[Lemma~3.3]{kruse2017}. Compare further with \cite{gyoengy1998}. 

\begin{corollary}
  \label{cor:pathwise}
  Let Assumptions~\ref{as:X0} to \ref{as:g} be satisfied 
  with $p \in [2,\infty)$ and $\gamma \in (0,1]$. Let $(\pi_{h^{(m)}})_{m \in
  \N} \subset [0,T]$ be a sequence of temporal grids with corresponding 
  maximum step sizes $|h^{(m)}|$ satisfying $\sum_{m = 1}^\infty |h^{(m)}| <
  \infty$. Then, there exist a random variable $m_0 \colon \Omega \to \N_0$ and
  a measurable set $A \in \F$ with $\P(A) = 1$ such that for all $\omega \in A$
  and $m \ge m_0(\omega)$ we have
  \begin{align*}
    \max_{n \in \{0,1,\ldots,N_{h^{(m)}}\}} \big| X_{h^{(m)}}^n(\omega) 
    - X(t_n, \omega) \big| \le |h^{(m)}|^{\min(\frac{1}{2} + \gamma, 1) -
    \frac{1}{p}}. 
  \end{align*}
\end{corollary}

%%%%%%%%%%%%%%%%%%%%%%%%%%%%%%%%%%
% Section on randomized quadrature rules
%%%%%%%%%%%%%%%%%%%%%%%%%%%%%%%%%%

\section{A randomized quadrature rule for stochastic processes}
\label{sec:quadrature}
In this section we introduce a randomized quadrature rule for integrals of 
stochastic processes, which is an essential ingredient in the error analysis of
the randomized Milstein method. It is based on a well-known 
variance reduction technique from Monte Carlo integration, the stratified
sampling. In dependence of the temporal regularity of the stochastic process
this technique is known to admit higher order convergence results than
the standard rate $\frac{1}{2}$ usually known for Monte Carlo methods.
Our result is an extension of results from
\cite{haber1966, haber1967} to stochastic processes.
Compare further with \cite{kruse2017} for a more
recent exposition of the deterministic case. 

In the following we consider an arbitrary
stochastic process $Y \colon [0,T]\times \Omega_W \to \R^d$ on 
the probability space $(\Omega_W, \F^W, \P_W)$ satisfying
$\|Y\|_{L^p([0,T]\times\Omega_W;\R^d)}<\infty$ for some $p \in 
[2,\infty)$. Let $\pi_h = \{t_j \, :\, j = 0,1,\ldots,N_h\}
\subset [0,T]$ be an arbitrary temporal grid with associated vector of step
sizes $h = (h_j)_{j = 1}^{N_h}$ as defined in \eqref{eq:stepsizes}. 
Recall that $|h|$ denotes the maximum step size in $\pi_h$.

Then, the goal is to give a numerical approximation
of the random variables
\begin{align*}
  \int_0^{t_n} Y(s) \diff{s} \in L^p(\Omega_W;\R^d)
\end{align*}
for each $n \in \{1,\ldots,N_h\}$. To this end we introduce the following
\emph{randomized Riemann sum approximation} $Q_{\tau,h}^n[Y]$ of
$\int_0^{t_n} Y(s) \diff{s}$ given by 
\begin{align}
    \label{eq:defQtau}
    Q_{\tau,h}^n[Y] := \sum_{j = 1}^n h_j  Y(t_{j-1} +  \tau_j h_j),
    \quad n \in \{1,\ldots,N_h\}, 
\end{align}
where $(\tau_j)_{j \in \N}$ is an independent family of
$\mathcal{U}(0,1)$-distributed random variables on the probability space
$(\Omega_\tau,\F^\tau,\P_\tau)$. In particular, we assume that the family
$(\tau_j)_{j \in \N}$ is independent of the stochastic process $Y$. 
Consequently, $Q_{\tau,h}^n[Y]$ is a random variable on
the product probability space $(\Omega,\F,\P)$ defined in \eqref{eq:Omega}.
For the formulation of the following theorem, we recall from
Section~\ref{sec:notation} that $\E_\tau[ \cdot]$ denotes the expectation with
respect to the measure $\P_\tau$.

\begin{theorem}
  \label{th4:randRiemann}
  For $p \in [2,\infty)$ let $Y \colon [0,T] \times \Omega_W \to \R^d$ 
  be a stochastic process with $Y \in L^p([0,T]\times\Omega_W;\R^d)$.
  Then, for every temporal grid $\pi_h$ and $n \in \{1,\ldots,N_h\}$
  the randomized Riemann sum approximation $Q_{\tau,h}^n[Y] \in
  L^p(\Omega ;\R^d)$ defined in \eqref{eq:defQtau}
  is an unbiased estimator for the integral $\int_0^{t_n} Y(s) \diff{s}$
  in the sense that 
  \begin{equation}
    \label{eq:mean0}
    \E_\tau \big[ Q_{\tau,h}^n[Y] \big]  = \int_{0}^{t_n} Y(s) \diff{s} \in
    L^p(\Omega_W;\R^d). 
  \end{equation}  
  Moreover, it holds true that 
  \begin{align}
    \label{eq4:errRie1}
    \begin{split}
    &\Big\| \max_{n \in \{1,\ldots,N_h\}}
    \Big| Q_{\tau,h}^n[Y] - \int_0^{t_n} Y(s) \diff{s} \Big| \,
    \Big\|_{L^p(\Omega)}
    \\ &\quad 
    \le 2 C_p T^{\frac{p-2}{2p}} \| Y
    \|_{L^p([0,T]\times\Omega_W;\R^d)} |h|^{\frac{1}{2}},
    \end{split}
  \end{align}
  where $C_p$ is a constant only depending on $p \in [2,\infty)$. 
  
  In addition, if $Y \in \mathcal{C}^\gamma([0,T],L^p(\Omega_W;\R^d))$ for some
  $\gamma \in (0,1]$, then we have   
  \begin{align}
    \label{eq4:errRie2}
    \begin{split}
      & \Big\| \max_{n \in \{1,\ldots,N_h\}} \Big|
      Q_{\tau,h}^n[Y] - \int_0^{t_n} Y(s) \diff{s}
      \Big| \, \Big\|_{L^p( \Omega)}
      \\ &\quad 
      \le C_p \sqrt{T} \| Y
      \|_{\mathcal{C}^\gamma([0,T],L^p( \Omega_W;\R^d))}
      |h|^{\frac{1}{2} + \gamma}, 
    \end{split}
  \end{align}
  where $C_p$ is the same constant
  as in \eqref{eq4:errRie1}.
\end{theorem}

\begin{proof}
  Since $Y\in L^p([0,T]\times\Omega_W;\R^d)$ there exists a null
  set $\mathcal{N}_0 \in \F^W$ such that for all $\omega \in \mathcal{N}_0^c =
  \Omega_W \setminus \mathcal{N}_0$ we have $\int_{0}^T
  |Y(s,\omega)|^p\diff{s}<\infty$.  
  Let us therefore fix an arbitrary realization $\omega \in \mathcal{N}_0^c$.
  Then for every $j \in \{1,\ldots,N_h\}$ we obtain
  \begin{align*}
    \int_{t_{j-1}}^{t_j} Y\big(s,\omega\big) \diff{s}
    = h_j \int_{0}^{1} Y \big(t_{j-1} + s h_j,\omega\big) \diff{s}
    = h_j \E_{\tau}[ Y(t_{j-1}+\tau_j h_j,\omega)],
  \end{align*}
  due to $\tau_j \sim \mathcal{U}(0,1)$. This immediately implies
  \eqref{eq:mean0} as well as $h_j Y(t_{j-1}+\tau_j h_j) \in L^p(\Omega;\R^d)$
  for every $j \in \{1,\ldots,N_h\}$.
  
  Next, we define a discrete-time
  error process $(E^n)_{n \in \{0,1,\ldots,N_h\}}$ by setting $E^0 \equiv 0$.
  Further, for every $n \in \{1,\ldots,N_h\}$ we set
  \begin{equation*}
    E^n := Q_{\tau,h}^n[Y] - \int_{0}^{t_n} Y (s) \diff{s}
    = \sum_{j=1}^n \Big( h_j Y( t_{j-1} + \tau_j h_j ) - \int_{t_{j-1}}^{t_j} Y
    (s) \diff{s}\Big),
  \end{equation*}
  which is evidently an $\R^d$-valued random variable on the product
  probability space $(\Omega, \F, \P)$. In particular, $(E^n)_{n \in
  \{0,1,\ldots,N_h\}} \subset L^p(\Omega;\R^d)$. Moreover,
  for each fixed $\omega \in \mathcal{N}_0^c$ we have
  that $E^n(\omega,\cdot) \colon \Omega_\tau \to \R^d$ is
  $\F_n^\tau$-measurable. Further, for
  each pair of $n,m \in \N$ with $0 \le m \le n \le N_h$ it holds true that
  \begin{align*}
    &\E_\tau[ E^n(\omega, \cdot)-E^m(\omega, \cdot) | \F^{\tau}_{m}]\\
    &\quad = \sum_{j=m+1}^n \E_\tau \Big[ h_j Y( t_{j-1} + \tau_j h_j, \omega )
    - \int_{t_{j-1}}^{t_j} Y(s, \omega) \diff{s} \Big| \F^\tau_m \Big]\\
    &\quad = \sum_{j=m+1}^n \E_\tau
    \big[ h_j Y( t_{j-1} + \tau_j h_j ) \big] - 
    \int_{t_m}^{t_n} Y (s, \omega) \diff{s} = 0, 
  \end{align*}
  since $\tau_j$ is independent of $\F^\tau_m$ for every $j > m$.
  Consequently, for every $\omega \in \mathcal{N}_0^c$ 
  the error process $(E^n(\omega, \cdot))_{n \in
  \{0,1,\ldots,N_h\}}$ is an $(\F^\tau_n)_{n \in
  \{0,1,\ldots,N_h\}}$-adapted $L^p(\Omega_\tau;\R^d)$-martingale.
  Thus,  the discrete-time version of the Burkholder-Davis-Gundy inequality 
  (see Theorem \ref{th:discreteBDG}) is applicable and yields  
  \begin{align*}
    \big\| \max_{n \in \{0,1,\ldots,N_h\}} | E^n(\omega, \cdot)|
    \big\|_{L^p(\Omega_\tau)}
    %_{L^p(\Omega;\R)}
    \le C_p \big\| [E(\omega,\cdot)]^{\frac{1}{2}}_{N_h}
    \big\|_{L^p(\Omega_\tau)}
    %{L^p(\Omega;\R)}.
    \quad \text{ for every } \omega \in \mathcal{N}_0^c.
  \end{align*}
  After inserting the quadratic variation $[E(\omega,\cdot)]_{N_h}$, taking the
  $p$-th power and integrating with respect to $\P_W$ we arrive at
  \begin{align}
    \label{eq4:est_E}
    \begin{split}
      &\big\| \max_{n \in \{0,1,\ldots,N_h\}} | E^n|\big\|_{L^p(\Omega)}^p
      = \int_{\Omega_W} \big\| \max_{n \in \{0,1,\ldots,N_h\}} | E^n(\omega,
      \cdot)| \big\|_{L^p(\Omega_\tau)}^p \diff{\P}_W(\omega) \\
      &\quad \le C_p^p \int_{\Omega_W} \Big\| \Big( \sum_{j = 1}^{N_h}
      \Big| \int_{t_{j-1}}^{t_j} \big( Y(t_{j-1} +\tau_j h_j,\omega)
      - Y(s,\omega ) \big) \diff{s}  \Big|^2 \Big)^{\frac{1}{2}} 
      \Big\|_{L^p(\Omega_\tau)}^p \diff{\P}_W(\omega)\\     
      &\quad = C_p^p \Big\| \sum_{j = 1}^{N_h} \Big| \int_{t_{j-1}}^{t_j}  \big(
      Y(t_{j-1}+\tau_j h_j) - Y(s) \big) \diff{s} \Big|^2 \,
      \Big\|_{L^{\frac{p}{2}}(\Omega)}^{\frac{p}{2}}\\
      &\quad \le C_p^p \Big( \sum_{j = 1}^{N_h} \Big\| \int_{t_{j-1}}^{t_j}
       \big| Y(t_{j-1}+\tau_jh_j ) - Y(s) \big| \diff{s}
       \Big\|_{L^p(\Omega)}^2  
      \Big)^{\frac{p}{2}},
    \end{split}
  \end{align}
  where the last step follows from an application of the triangle
  inequality for the $L^{\frac{p}{2}}(\Omega)$-norm.
  Now, after taking the $p$-th root, a further application of the triangle
  inequality yields
  \begin{align}
    \label{eq4:est_Ef}
    \begin{split}
      \big\| \max_{n \in \{1,\ldots,N_h\}} | E^n| \big\|_{L^p(\Omega)}
      &\le C_p \Big( \sum_{j = 1}^{N_h} \Big\| \int_{t_{j-1}}^{t_j}
      | Y(s) | \diff{s} \Big\|^2_{L^p(\Omega)} \Big)^{\frac{1}{2}} \\
      &\quad + C_p \Big( \sum_{j = 1}^{N_h} h_j^2 \big\|
      Y(t_{j-1} +  h_j \tau_j) \big\|_{L^{p}(\Omega;\R^d)}^2
      \Big)^{\frac{1}{2}}.
    \end{split}
  \end{align}
  The first term on the right hand side of \eqref{eq4:est_Ef} is then bounded 
  by an application of H\"older's inequality as follows
  \begin{align}
    \label{eq4:term1}
    \begin{split}
      \Big( \sum_{j = 1}^{N_h} \Big\| \int_{t_{j-1}}^{t_j}
      | Y(s) | \diff{s} \Big\|^2_{L^p(\Omega)} \Big)^{\frac{1}{2}}
      &=\Big( \sum_{j = 1}^{N_h} \Big(\E_W \Big[\Big( \int_{t_{j-1}}^{t_j}
      |Y(s)| \diff{s}\Big)^p\Big]\Big)^{\frac{2}{p}} \Big)^{\frac{1}{2}}\\
      &\leq  \Big(  \sum_{j = 1}^{N_h} h_j^{2-\frac{2}{p}} \Big(
      \int_{t_{j-1}}^{t_j} 
      \E_W [|Y(s)|^p] \diff{s}\Big)^{\frac{2}{p}} \Big)^{\frac{1}{2}}. 
    \end{split}
  \end{align}
  Now, if $p = 2$ we directly obtain the desired estimate
  \begin{align*}
    \Big( \sum_{j = 1}^{N_h} \Big\| \int_{t_{j-1}}^{t_j}
      | Y(s) | \diff{s} \Big\|^2_{L^2(\Omega)} \Big)^{\frac{1}{2}}
      \le |h|^{\frac{1}{2}} \|Y\|_{L^{2}([0,T]\times\Omega_W;\R^d)}.
  \end{align*}
  For $p \in (2,\infty)$ the estimate in \eqref{eq4:term1} is completed 
  by a further application of H\"older's inequality 
  with conjugated exponents $\rho = \frac{p}{2} \in (1,\infty)$ and $\rho' =
  \frac{p}{p-2}$. This yields
  \begin{align}
    \label{eq4:term2}
    \begin{split}
    & \Big(  \sum_{j = 1}^{N_h} h_j^{2-\frac{2}{p}} \Big( \int_{t_{j-1}}^{t_j}
      \E_W [|Y(s)|^p] \diff{s}\Big)^{\frac{2}{p}} \Big)^{\frac{1}{2}}\\
    &\quad \leq  \Big(\sum_{j=1}^{N_h} h_j^{\rho'(2-\frac{2}{p})}
      \Big)^{\frac{1}{2\rho'}} 
    \Big(\sum_{j=1}^{N_h}\int_{t_{j-1}}^{t_j} 
      \E_W [|Y(s)|^p] \diff{s}\Big)^{\frac{1}{p}}\\
    &\quad \leq  T^{\frac{p-2}{2p}} |h|^{\frac{1}{2}}
      \|Y\|_{L^{p}([0,T]\times\Omega_W;\R^d)}
    \end{split}
  \end{align}
  as claimed, since $T^{\frac{1}{2 \rho'}} =
  T^{\frac{p - 2}{2 p}}$ as well as $|h|^{\frac{1}{2}(
  2 - \frac{2}{p}) - \frac{1}{2\rho'}} = |h|^{\frac{1}{2}}$.

  In the same way we obtain an estimate for the second term on
  the right hand side of \eqref{eq4:est_Ef} by additionally taking note 
  of the fact that
  \begin{align}
    \label{eq4:equivalent}
    \begin{split}
      h_j^2 \|Y(t_{j-1}+\tau_jh_j)\|_{L^{p}(\Omega;\R^d)}^2
      &= h_j^{2 - \frac{2}{p}} \big( h_j \E_W \big[ \E_\tau [ |Y(t_{j-1}
      +\tau_j h_j) |^p ] \big] \big)^{\frac{2}{p}}\\ 
      &= h_j^{2 - \frac{2}{p}} \Big( \E_{W} \Big[ \int_{t_{j-1}}^{t_j} |Y(s)|^p
      \diff{s} \Big] \Big)^{\frac{2}{p}}.
    \end{split}
  \end{align}
  Then, one proceeds as in \eqref{eq4:term1} and \eqref{eq4:term2}.
  Altogether, \eqref{eq4:est_E}, \eqref{eq4:term1}, and 
  \eqref{eq4:equivalent} yield
  \begin{align*}
     &\big\| \max_{n \in \{1,\ldots,N_h\}} | E^n|
    \big\|_{L^p(\Omega)} \le 2 C_p T^{\frac{p-2}{2p}} \| Y
    \|_{L^p([0,T]\times\Omega_W;\R^d)} |h|^{\frac{1}{2}}.
  \end{align*}
  This completes the proof of \eqref{eq4:errRie1}.

  Next, if $Y \in \mathcal{C}^\gamma([0,T],L^{p}(\Omega_W;\R^d))$
  we can improve the estimate in \eqref{eq4:est_E} by 
  \begin{align*}
    &\Big\| \int_{t_{j-1}}^{t_j} \big| Y(t_{j-1} +  h_j \tau_j) - Y(s)
    \big| \diff{s}  \Big\|_{L^{p}(\Omega)}\\
    &\quad \le \int_{t_{j-1}}^{t_j} \big(\E_\tau \big[ \E_W \big[ \big| 
    Y(t_{j-1} +  h_j \tau_j) - Y(s) \big|^p \big] \big] \big)^{\frac{1}{p}} 
    \diff{s}\\  
    &\quad \le \| Y \|_{\mathcal{C}^\gamma([0,T],L^{p}(\Omega_W;\R^d))}
    \int_{t_{j-1}}^{t_j} \big(\E_\tau\big[|t_{j-1} + \tau_j  h_j - s|^{\gamma
    p}\big]\big)^{\frac{1}{p}} \diff{s}\\ 
    &\quad \le \| Y \|_{\mathcal{C}^\gamma([0,T],L^{p}(\Omega_W;\R^d))}
    h_j^{1 + \gamma}.  
  \end{align*}
  Thus, inserting this into \eqref{eq4:est_E} gives
  \begin{align*}
    \big\| \max_{n \in \{0,1,\ldots,N_h\}} | E^n| \big\|_{L^p(\Omega)}
    &\le C_p \Big( \sum_{j = 1}^{N_h} \| Y
    \|_{\mathcal{C}^\gamma([0,T],L^{p}(\Omega_W;\R^d))}^2 
    h_j^{2(1 + \gamma)} \Big)^{\frac{1}{2}}\\
    &\le C_p T^{\frac{1}{2}} \| Y
    \|_{\mathcal{C}^\gamma([0,T],L^{p}(\Omega_W;\R^d))} 
    |h|^{\frac{1}{2} + \gamma}.
  \end{align*}
  This completes the proof of \eqref{eq4:errRie2}.
\end{proof}

%%%%%%%%%%%%%%%%%%%%%%%%%%%%%%%%%%
% Section on bistability
%%%%%%%%%%%%%%%%%%%%%%%%%%%%%%%%%%

\section{Stability of the drift-randomized Milstein method}
\label{sec:bistab}

In this section we show that the randomized Milstein method constitutes a
stable numerical method. More precisely, we consider the notion of
\emph{stochastic bistability} that has been introduced in 
\cite{beyn2010, kruse2012, kruse2014b} and is based on 
the abstract framework for discrete approximations developed by
\cite{stummel1973}.

For the introduction of the bistability concept
let $\pi_h$ be an arbitrary temporal grid.
It is then convenient to introduce the space $\mathcal{G}_h^p
:= \mathcal{G}(\pi_h, L^p(\Omega;\R^d))$ of all 
$(\mathcal{F}^h_{n})_{n \in \{0,1,\ldots,N_h\}}$-adapted
and $\R^d$-valued stochastic grid functions, where 
the discrete-time filtration $(\mathcal{F}^h_{n})_{n \in \{0,1,\ldots,N_h\}}$
associated to $\pi_h$ has been defined in \eqref{eq:filtration}. More formally,
we set
\begin{align*}
  \mathcal{G}_h^p := \big\{ (Y_h^n)_{n = 0}^{N_h} \, : \, Y_h^n \in L^p(\Omega,
  \F_n^h, \P; \R^d) \text{ for each
  } n \in \{0,1,\ldots,N_h\} \, \big\}.
\end{align*}
We endow the space $\mathcal{G}_h^p$ with the norm
\begin{align*}
  \big\| Y_h \big\|_{p,\infty} := 
  \big\|\max_{n \in \{0,1,\ldots,N_h\}} |Y_{h}^n| \big\|_{L^p(\Omega)}, 
  \quad Y_h \in \mathcal{G}_h^p.
\end{align*}
Then, the tuple $G_h:= (\mathcal{G}_h^p, \| \cdot \|_{p,\infty})$
becomes a Banach space. Before we continue let us briefly take note of the
fact that the error in Definition~\ref{def:strerr} is in fact measured
in terms of the norm $\| \cdot \|_{p,\infty}$. To be more precise, we have
\begin{align*}
  \big\| X_h - X|_{\pi_h} \big\|_{p,\infty}
  = \big\| \max_{n \in \{0,1,\ldots, N_h\} } | X_h^n - X(t_n)| 
  \big\|_{L^p(\Omega)},
\end{align*}
where $X_h = (X_h^n)_{n = 0}^{N_h} \in \mathcal{G}_h^p$ denotes the stochastic
grid function generated by the numerical scheme \eqref{eq:RandM} on $\pi_h$. In
addition, $X|_{\pi_h}$ denotes the restriction of the exact
solution $X$ of the SDE \eqref{eq:SODE} to the temporal grid points in $\pi_h$.
Theorem~\ref{th:Lp_est} then ensures that indeed $X|_{\pi_h} \in
\mathcal{G}_h^p$, where $p \in [2,\infty)$ is determined by
Assumption~\ref{as:X0}.

The main idea of the bistability concept is now to relate the global error
$X_h - X|_{\pi_h}$ to certain estimates of the local truncation error defined
in \eqref{eq:consistency} below. In order to obtain
optimal error estimates it is however crucial to measure
the local errors in a modified norm. Here, we follow an approach developed in 
\cite{beyn2010, kruse2012} and introduce the so called \emph{stochastic Spijker
norm} on $\mathcal{G}_h^p$ given by
\begin{align}
  \label{eq:Spijker}
  \big\| Z_h \big\|_{S,p} := 
  \| Z_{h}^0 \|_{L^p(\Omega;\R^d)}
  + \Big\| \max_{n\in \{1,2,\ldots,N_h\}} \big| \sum_{j=1}^{n}
  Z_{h}^j \big| \Big\|_{L^p(\Omega)}.
\end{align}
This gives rise to a further Banach space denoted by $G_h^S = (\mathcal{G}_h^p,
\| \cdot\|_{S,p} )$. Note that deterministic versions of this norm are used in
numerical analysis for finite difference methods, see for instance
\cite{spijker1968, spijker1971, stummel1973}. For a more detailed discussion in
the context of SDEs we refer the reader to \cite{beyn2010}.

\begin{remark}
  In the following, we choose the value of the parameter $p \in [2,\infty)$ in
  the definition of the spaces $G_h$ and $G^S_h$ to be the same as in
  Assumption~\ref{as:X0}.

  Moreover, for every fixed temporal grid $\pi_h$ the norms
  $\| \cdot \|_{p,\infty}$ and $\| \cdot \|_{S,p}$ are easily seen to be
  equivalent. However, the norm of the embedding $G_h \hookrightarrow G_h^S$
  grows with the number of steps $N_h$ in $\pi_h$. Thus, the topology generated
  by the Spijker norm in the limit $|h| \to 0$ is stronger in the following
  sense: Let $(\pi_h^{(j)})_{j \in \N}$ be a 
  sequence of temporal grids with 
  $|h^{(j)}| \to 0$ for $j \to \infty$. Then, if $(Z_h^{(j)})_{j \in \N}
  \subset G_{h^{(j)}}^S$ is a sequence of stochastic grid functions with
  \begin{align*}
    \lim_{j \to \infty} \| Z_h^{(j)} \|_{S,p} = 0,
  \end{align*}
  the same holds true with respect to the $\| \cdot
  \|_{p,\infty}$-norm, since $\| Z_h^{(j)} \|_{p,\infty} \le 2 \| Z_h^{(j)}
  \|_{S,p}$ for all $j \in \N$. In general, the converse implication is,
  \emph{wrong}.
\end{remark}

We are now in a position to state the definition of bistability.
  
\begin{definition}
  \label{def:bistability}  
  The numerical method \eqref{eq:RandM} is called (stochastically) 
  \emph{bistable} if there exist constants $C_1, C_2 \in (0,\infty)$
  and $p \in [2, \infty)$ such that for every temporal grid $\pi_h$ 
  with $|h|\leq h_0:=\min(1,T)$ and all $Y_h \in G_h$ it holds true that
  \begin{equation}
    \label{eq:bistability}
    C_1 \| R_h \|_{S,p} \leq \| X_h - Y_h \|_{p,\infty} 
    \leq C_2 \| R_h  \|_{S,p},
  \end{equation} 
  where $X_h \in G_h$ is generated by \eqref{eq:RandM} and $R_h = R_h[Y_h] \in
  G^S_h$ denotes the \emph{residual} of $Y_h$ given by $R_h^0 = Y_h^0 - X_h^0$
  and  
  \begin{equation}
    \label{eq:residual}
    R_h^j := Y_h^j - Y_h^{j-1} - \Phi^j_h(Y^{j-1}_h,\tau_j)
  \end{equation}
  for all $j \in \{1,\ldots,N_h\}$.
\end{definition}

\begin{remark}
  \label{rmk:numericalsolution} 
    (i) The properties of the increment function $\Phi^j_h$ (see
    Lemma~\ref{lem:Phi_prop}) ensure that $R_h = R_h[Y_h]
    \in G^S_h$ if $Y_h \in G_h$. Therefore, the norms in \eqref{eq:bistability}
    are well-defined for every $Y_h \in G_h$.

    (ii) If a numerical method is bistable, then \eqref{eq:bistability}
      says that we can estimate the $\|\cdot\|_{p,\infty}$-difference between
      $X_h$ and an arbitrary stochastic
      grid function in terms of the residual of that
      grid function. Here the residual \eqref{eq:residual}
      measures how well $Y_h \in G_h$ satisfies the recursion
      \eqref{eq:transformM1} defining the numerical method. In addition, 
      the first inequality in \eqref{eq:bistability} shows that the Spijker
      norm yields asymptotically optimal error estimates.
    
      (iii) For the proof of Theorem~\ref{th:main2} we will 
      apply the inequality \eqref{eq:bistability} with $Y_h :=
      X|_{\pi_h}$ in Section~\ref{sec:consistency}. However, the connection
      between Definition~\ref{def:bistability} and the general notion of
      \emph{stability} used in numerical analysis is
      that we also easily estimate the influence of small perturbations to the
      numerical method. For instance, let $\rho_h = ( \rho^n_h )_{n = 0}^{N_h}
      \in G^S_h$ model the inevitable round-off errors occurring during the
      computation of $X_h$ on a computer. That is, instead of $X_h$ we actually
      only observe $\tilde{X}_h = ( \tilde{X}_h^n )_{n = 0}^{N_h}$ in practice,
      where $\tilde{X}_h^0 = X_h^0 + \rho^0_h$ and 
      \begin{align*}
        \tilde{X}_h^j = \tilde{X}_h^{j-1} + \Phi^j_h(\tilde{X}^{j-1}_h,\tau_j)
        + \rho^j_h
      \end{align*}
      for all $j \in \{1,\ldots,N_h\}$. Then, the bistability inequality
      \eqref{eq:bistability} shows that
      \begin{align*}
        C_1 \|\rho_h\|_{S,p} \le \| X_h - \tilde{X}_h \|_{p,\infty} 
        \le C_2 \|\rho_h\|_{S,p}.
      \end{align*}
      For example, for the implementation of an implicit and bistable numerical
      method, it is not necessary to solve exactly the implicit nonlinear
      equations defining the numerical method. An approximation by, for
      instance, Newton's method is sufficient as long as the additional errors
      measured in the Spijker norm are of the same (asymptotic) order as the
      global error. 
\end{remark}

The remainder of this section is devoted to the proof that under
Assumptions~\ref{as:f} and \ref{as:g} the drift-randomized
Milstein method \eqref{eq:RandM} is indeed bistable, see
Theorem~\ref{th:Bistability} further below. For the proof the following lemma
will be useful.

\begin{lemma}
  \label{lm:stability}
  Let Assumptions~\ref{as:f} and \ref{as:g} be satisfied.
  Let $\pi_h$ be an arbitrary temporal grid with $|h|\leq \min(1,T)$. 
  Then, for all stochastic grid functions $Y_h, Z_h \in \mathcal{G}_h^p$, $p
  \in[2,\infty)$, and $k
  \in \{1,\ldots,N_h\}$ it holds true that
  \begin{align}
    \label{eq:stabilityforPhi}
    \begin{split}
      & \Big\|\max_{n \in \{1,\ldots,k\}}
      \Big| \sum_{j=1}^n \big( \Phi^j_h(Y^{j-1}_h,\tau_j) 
      - \Phi^j_h(Z^{j-1}_h,\tau_j) \big) \Big| \, \Big\|_{L^p(\Omega)} \\
      &\quad \le C_3 \Big( \sum_{j=1}^{k} h_j \Big\|
      \max_{i \in \{0, \ldots, j-1\} }
      \big| Y^{i}_h - Z_h^i \big| \Big\|^2_{L^p(\Omega)}
      \Big)^{\frac{1}{2}},
    \end{split}
  \end{align}
  where $C_{3}= K_f \big( 1 + K_f + 2 m K_g C_p \big) \sqrt{T}
  + K_g m C^2_p( 1  + m C_p) $. Furthermore, with $C_4 = C_3 \sqrt{T}$
  \begin{align}
    \label{eq:stabilitysquared}
    \Big\|\max_{n \in \{1,\ldots,N_h\} }
    \Big| \sum_{j=1}^n \big( \Phi^j_h(Y^{j-1}_h,\tau_j) 
    - \Phi^j_h(Z^{j-1}_h,\tau_j) \big) 
    \Big| \Big\|_{L^p(\Omega)}
    \leq  C_4\Big\|Y_h-Z_h\Big\|_{p,\infty}. 
  \end{align}
\end{lemma}

\begin{proof}
  Recalling the definitions of $\Phi_h^j$ and $\Psi_h^j$ 
  from \eqref{eq:PhiM_expression} and \eqref{eq:PsiM} we have
  \begin{align*}
    &\Phi^j_h(Y^{j-1}_h,\tau_j) - \Phi^j_h(Z^{j-1}_h, \tau_j)\\
    &\quad = h_j \big( 
    f( t_{j-1} + \tau_j h_j, \Psi_h^j(Y_h^{j-1},\tau_j) )
    - f( t_{j-1} + \tau_j h_j, \Psi_h^j( Z_h^{j-1},\tau_j) ) \big)\\
    &\qquad +  \sum_{r=1}^m \big( g^r(t_{j-1},Y^{j-1}_h) -  
    g^r(t_{j-1},Z^{j-1}_h) \big) I^{t_{j-1},t_j}_{(r)}\\
    &\qquad + \sum_{r_1, r_2 = 1}^{m} \big( g^{r_1,r_2}(t_{j-1}, Y^{j-1}_h) - 
    g^{r_1,r_2}(t_{j-1},Z^{j-1}_h) \big) I_{(r_2,r_1)}^{t_{j-1},t_j}\\
    &\quad =: \Xi_{1}^j+\Xi_{2}^j+\Xi_{3}^j.
  \end{align*}
  We estimate the three terms separately. For the estimate of
  $\Xi_1^j$ in the stochastic Spijker norm we first apply Assumption~\ref{as:f}
  and obtain for every $k \in \{1,\ldots,N_h\}$ 
  \begin{align*}
    \Big\|\max_{n \in \{1,\ldots,k\}} \big|\sum_{j=1}^n\Xi_{1}^j \big|
    \Big\|_{L^p(\Omega)} 
    &\leq \sum_{j=1}^{k} \|\Xi_{1}^j \| _{L^p(\Omega;\R^d)} \\
    &\leq K_f \sum_{j = 1}^k h_j 
    \big\| \Psi_h^j(Y_h^{j-1},\tau_j) - \Psi_h^j( Z_h^{j-1},\tau_j)
    \big\|_{L^p(\Omega;\R^d)}.
  \end{align*}
  In light of Assumption~\ref{as:f}, the Lipschitz continuity \eqref{eq:Lip_g} 
  of $g^r$, and that
  the increment $I^{t_{j-1},t_{j-1}+\tau_j h_j}_{(r)}$ is independent
  of $Y_h^{j-1}$ and $Z_h^{j-1}$ we further have
  \begin{align*}
    &\big\| \Psi_h^j(Y_h^{j-1},\tau_j) - \Psi_h^j( Z_h^{j-1},\tau_j)
    \big\|_{L^p(\Omega;\R^d)}\\
    &\quad \le (1 + K_f h_j) \big\| Y_h^{j-1} - Z_h^{j-1}
    \big\|_{L^p(\Omega;\R^d)}\\ 
    &\qquad + \sum_{r=1}^{m} \big\| g^r (t_{j-1}, Y^{j-1}_h ) - g^r( t_{j-1},
    Z^{j-1}_h ) \big\|_{L^p(\Omega;\R^d)} \big\|I^{t_{j-1},t_{j-1}+\tau_j
    h_j}_{(r)}\big\|_{L^p(\Omega)}\\
    &\quad \le \big( 1 + K_f |h| + m  K_g C_p |h|^{\frac{1}{2}}\big)
    \Big\| \max_{i \in \{0,\ldots,j-1\} }
    \big| Y_h^{i} - Z_h^{i}\big| \Big\|_{L^p(\Omega)},
  \end{align*}
  where the last step follows from (\ref{eq5:est_incr}). After taking squares,
  applying the Cauchy-Schwarz inequality
  and $|h| \le 1$ we arrive at
  \begin{align}
    \label{eq5:estXi1}
    \begin{split}
      &\Big\|\max_{n \in \{1,\ldots,k\}} \big|\sum_{j=1}^n\Xi_{1}^j \big|
      \Big\|_{L^p(\Omega)}^2 \\
      &\quad \le K_f^2 \big( 1 + K_f + m  K_g C_p \big)^2
      T \sum_{j = 1}^k h_j \Big\| \max_{i \in \{0,\ldots,j-1\} }
      \big| Y_h^{i} - Z_h^{i}\big| \Big\|_{L^p(\Omega)}^2. 
    \end{split}
  \end{align}
  For the estimate of $\Xi_{2}$
  first note that $(M^j)_{j = 0}^{N_h}$ defined by $M^0 = 0$ and
  \begin{align*}
    M^n := \sum_{j = 1}^{n} \Xi_{2}^j, \quad \text{ for } n \in
    \{1,\ldots,N_h\},
  \end{align*}
  is a discrete-time martingale with respect to the filtration
  $(\F^W_{t_j}\otimes \F^{\tau}_j)_{j\in \{0,1,\ldots,N_h\}}$.
  Hence, an application of Theorem~\ref{th:discreteBDG}
  gives
  \begin{align*}
    &\Big\|\max_{n \in \{1,\ldots,k\} }
    \big| \sum_{j=1}^n \Xi_{2}^j \big| \Big\|^2_{L^p(\Omega)}
    = \Big\|\max_{n \in \{1,\ldots,k\} } \big| M^n \big|
    \Big\|^2_{L^p(\Omega)} 
    \le C_p^2 \big\| [M]^{\frac{1}{2}} \big\|_{L^p(\Omega)}^2. 
  \end{align*}
  After inserting the quadratic variation of $M$ we therefore 
  obtain the estimate 
  \begin{align*}
    &\Big\|\max_{n \in \{1,\ldots,k\} }
    \big| \sum_{j=1}^n \Xi_{2}^j \big| \Big\|^2_{L^p(\Omega)}
    \le C_p^2 \Big\| \Big( \sum_{j = 1}^k \big| \Xi_{2}^j \big|^2
    \Big)^{\frac{1}{2}} \Big\|^2_{L^p(\Omega)} \\
    &\quad = C_p^2 \Big\| \sum_{j = 1}^k \Big| \sum_{r=1}^m 
    \big( g^r(t_{j-1},Y^{j-1}_h) - g^r(t_{j-1},Z^{j-1}_h) \big)
    I^{t_{j-1},t_j}_{(r)} \Big|^2 \Big\|_{L^{\frac{p}{2}}(\Omega)}\\
    &\quad \le C_p^2 \sum_{j = 1}^k \Big\| \sum_{r=1}^m
    \big| g^r(t_{j-1},Y^{j-1}_h) - g^r(t_{j-1},Z^{j-1}_h) \big|
    |I^{t_{j-1},t_j}_{(r)}| \Big\|_{L^{p}(\Omega)}^2.
  \end{align*}
  Making again use of the Lipschitz continuity \eqref{eq:Lip_g} 
  of $g^r$ and of the independence 
  of the increments $I^{t_{j-1},t_j}_{(r)}$ as well as its estimate
  \eqref{eq5:est_incr} finally yields
  \begin{align}
    \label{eq5:estXi2}
    \begin{split}
      \Big\|\max_{n \in \{1,\ldots,k\} }
      \big| \sum_{j=1}^n \Xi_{2}^j \big| \Big\|^2_{L^p(\Omega)} 
      &\le m K_g^2 C_p^2 \sum_{j = 1}^k \sum_{r=1}^m 
      \big\| Y^{j-1}_h - Z^{j-1}_h \big\|^2_{L^{p}(\Omega)}
      \big\|I^{t_{j-1},t_{j}}_{(r)}\big\|_{L^p(\Omega)}^2 \\
      &\le  m^2 K_g^2 C_p^4 \sum_{j=1}^{k}h_j 
      \Big\|\max_{i \in \{0,\ldots,j-1\} } 
      \big| Y^{i}_h - Z^{i}_h \big|\Big\|_{L^p(\Omega)}^2.
    \end{split}
  \end{align}
  The remaining term $\Xi_3$ is estimated analogously, since
  the iterated stochastic integrals $I_{(r_2,r_1)}^{t_{j-1},t_j}$ are also
  independent of $Y_h^{j-1}$, $Z_h^{j-1}$. By estimate
  (\ref{eq5:est_stochincr}) we obtain  
  \begin{align}
    \label{eq5:estXi3}
    \begin{split}
      \Big\| \max_{n \in \{1,\ldots,k\} } 
      \big| \sum_{j=1}^n \Xi_{3}^j \big| \Big\|^2_{L^p(\Omega)}
      \leq m^4 K_g^2 C_p^6 
      \sum_{j=1}^{k} h_j \Big\|\max_{i \in \{0,\ldots, j-1\} } 
      \big| Y^{i}_h - Z^{i}_h \big| \Big\|_{L^p(\Omega)}^2.     
    \end{split}
  \end{align}
  Combining the estimates \eqref{eq5:estXi1}, \eqref{eq5:estXi2}, and
  \eqref{eq5:estXi3}, completes the proof of \eqref{eq:stabilityforPhi}.

  Finally, the inequality \eqref{eq:stabilitysquared} is
  easily deduced from \eqref{eq:stabilityforPhi}. 
\end{proof}

\begin{theorem}
  \label{th:Bistability}
  Under Assumptions \ref{as:X0} to \ref{as:g} with $p \in [2,\infty)$
  the drift-randomized Milstein method \eqref{eq:RandM} is bistable with
  stability constants $C_1=\frac{1}{3+C_4}$ and $C_2=\sqrt{2}\ee^{C_3^2 T}$,
  where $C_3$ and $C_4$ are defined in Lemma~\ref{lm:stability}.
\end{theorem}

\begin{proof}
  Let $Y_h \in G_h$ be arbitrary. By recalling the definition of the residual
  $R_h = R_h[Y_h] \in G^S_h$ from \eqref{eq:residual} we get for every $n \in
  \{1,\ldots,N_h\}$
  \begin{align*}
      \sum_{j=1}^n  R_h^j
      &= \sum_{j=1}^n \big( Y_h^j - Y_h^{j-1} - \Phi^j_h( Y^{j-1}_h,\tau_j)
      \big)
      = Y_h^n - Y_h^0  - \sum_{j = 1}^n \Phi^j_h( Y^{j-1}_h,\tau_j). 
  \end{align*}
  Due to \eqref{eq:transformM1} we further have
  \begin{align*}
    X_h^n - X_h^0 - \sum_{j = 1}^n \Phi^j_h( X^{j-1}_h,\tau_j) = 0.
  \end{align*}
  Therefore, by a telescopic sum argument we obtain that
  \begin{align}
    \label{eq4:telescopic}
    \sum_{j=1}^n  R_h^j = \big( Y_h^n - X_h^n \big) - \big( Y_h^0 - X_h^0
    \big) - \sum_{j = 1}^n \big( \Phi^j_h( Y^{j-1}_h,\tau_j) 
    - \Phi^j_h( X^{j-1}_h,\tau_j) \big) .
  \end{align}
  Inserting this into the Spijker norm of the residual yields
  \begin{align*}
    \| R_h \|_{S,p}
    &= \| R_{h}^0 \|_{L^p(\Omega;\R^d)} + \Big\|\max_{n \in \{1,\ldots, N_h\} }
    \Big| \sum_{j=1}^n R^j_h \Big|\, \Big\|_{L^p(\Omega)}\\
    &\le 2 \| X_h^0 - Y_h^0 \|_{L^p(\Omega;\R^d)} 
    + \Big\| \max_{n \in \{1, \ldots, N_h\} }
    \big| X_h^n - Y_h^n \big| \Big\|_{L^p(\Omega)}\\
    &\quad + \Big\|\max_{n \in \{1,\ldots, N_h\} }
    \Big| \sum_{j = 1}^n \big( \Phi^j_h( X^{j-1}_h,\tau_j) 
    - \Phi^j_h( Y^{j-1}_h,\tau_j) \big) \Big| \, \Big\|_{L^p(\Omega)}\\
    &\le (3 + C_4) \big\| X_h - Y_h \big\|_{p,\infty},
  \end{align*}
  where the last step follows from an application of
  \eqref{eq:stabilitysquared}. Thus we have $C_1=\frac{1}{3+C_4}$. 

  On the other hand, by rearranging \eqref{eq4:telescopic} 
  the distance $|X_h^n-Y_h^n|$ can be represented for every $n \in
  \{1,\ldots,N_h\}$ by
  \begin{align*}
    |X_h^n-Y_h^n|
    &\le \Big| \sum_{j=1}^{n} \big( \Phi(X_h^{j-1},\tau_j) - 
    \Phi(Y_h^{j-1},\tau_j) \big) \Big| + |R_h^0| + 
    \Big|\sum_{j=1}^{n} R_h^j \Big|.
  \end{align*}
  Therefore, after taking the maximum over $n \in \{0,1,\ldots,k\}$
  with arbitrary $k \in \{1,\ldots,N_h\}$, applications of the
  squared $L^p(\Omega)$-norm and Lemma~\ref{lm:stability} then yield
  \begin{align*}
    &\big\| \max_{n \in \{0,1,\ldots,k\}} 
    |X_h^n - Y_h^n| \big\|_{L^p(\Omega)}^2 \\
    &\quad \le 2 \Big\| \max_{n \in \{1,\ldots,k\} }
    \Big| \sum_{j=1}^{n}\big(\Phi(X_h^{j-1},\tau_j)
    -\Phi(Y_h^{j-1},\tau_j)\big) \Big| \, \Big\|_{L^p(\Omega)}^2\\
    &\qquad + 2 \Big( \| R_h^0 \|_{L^p(\Omega;\R^d)}
    + \Big\|\max_{n\in \{1,\ldots,k\} }
    \Big| \sum_{j=1}^{n} R_h^j \Big| \, \Big\|_{L^p(\Omega)} \Big)^2\\ 
    &\quad \le 2 C_3^2 \sum_{j=1}^{k} h_j \Big\| 
    \max_{n\in \{0,\ldots,j-1\}} 
    \big| X^{n}_h - Y_h^n \big| \Big\|^2_{L^p(\Omega)}
    +2\|R_h\|_{S,p}^2. 
  \end{align*}
Now an application of
the discrete Gronwall inequality (Lemma~\ref{lem:Gronwall}) gives 
\begin{align*}
  \big\|\max_{n\in \{0,\ldots,N_h\}}|X_h^n-Y_h^n|\big\|_{L^p(\Omega)}^2
  \le 2 \| R_h\|_{S,p}^2 \exp\Big( 2 C_3^2 \sum_{j=1}^{N_h} h_j \Big), 
\end{align*}
where we can use the fact that $\sum_{j=1}^{N_h}h_j=T$. In total, we obtain
that 
\begin{align*}
\| X_h - Y_h \|_{p,\infty} \leq C_2 \| R_h\|_{S,p}, 
\end{align*}
with $C_2=\sqrt{2}\ee^{C_3^2 T}$. 
\end{proof}

%%%%%%%%%%%%%%%%%%%%%%%%%%%%%%%%%%
% Section on consistency
%%%%%%%%%%%%%%%%%%%%%%%%%%%%%%%%%%

\section{Consistency and convergence of the randomized Milstein method}
\label{sec:consistency}

In this section we show that the drift-randomized Milstein method
\eqref{eq:RandM} is strongly convergent of order $\min(\frac{1}{2} + \gamma,
1)$ as asserted in Theorem~\ref{th:main2}. To this end we first show that the
numerical method is consistent with the SDE \eqref{eq:SODE} in the following
sense. For the formulation of Definition~\ref{def:consistency} recall the
definitions of the Spijker norm $\| \cdot \|_{S,p}$ in \eqref{eq:Spijker} and
of the residual $R_h[Y_h]$ of a grid function $Y_h \in G_h$ in
\eqref{eq:residual}. 

\begin{definition}
  \label{def:consistency}
  The numerical method \eqref{eq:RandM} is called \emph{consistent of order
  $\beta \in (0,\infty)$} with the SDE \eqref{eq:SODE} 
  if there exist constants $C \in (0,\infty)$ and $p \in [2,\infty)$ such
  that for every temporal grid $\pi_h$ with $|h|\leq \min(1,T)$ we have
  \begin{equation}
    \label{eq:consistency}
    \big\| R_h[ X|_{\pi_h} ] \big\|_{S,p} \le C |h|^{\beta}
  \end{equation}
  where $X|_{\pi_h}$ is the restriction of the exact solution of
  \eqref{eq:SODE} to the temporal grid $\pi_h$. 
\end{definition}

Below we will show that the drift-randomized Milstein method
\eqref{eq:RandM} is consistent of order $\beta = \min(\frac{1}{2} + \gamma,
1)$ under Assumptions~\ref{as:X0} to \ref{as:g}. For this we first present some
estimates for the diffusion term.

\begin{lemma} 
  \label{lm:milstein1}
  Let Assumptions~\ref{as:X0} to \ref{as:g} be satisfied with
  $p \in [2,\infty)$ and $\gamma \in (0,1]$. Let $\pi_h$ be an arbitrary
  temporal grid with $|h|\leq \min(1,T)$. 
  For each $r\in\{1,2,\ldots,m\}$, $j \in \{1,\ldots,N_h\}$ let us
  denote by $\Gamma^j_{(r)}$ the following expression 
  \begin{align*}
    \Gamma^j_{(r)} 
    &= \int_{t_{j-1}}^{t_j} g^r\big(s,X(s)\big) \diff{W^r(s)} - 
    g^r\big(t_{j-1}, X(t_{j-1}) \big) I_{(r)}^{t_{j-1},t_j} \\
    &\qquad - \sum_{r_2=1}^m g^{r,r_2} \big(t_{j-1},X(t_{j-1})\big) 
    I_{(r_2,r)}^{t_{j-1},t_j}.
  \end{align*} 
  Then there exists $C \in (0,\infty)$ only depending 
  on $T$, $p$, $m$, $K_g$, and $\tilde{K}_f$ such that 
  \begin{equation*}
    %\label{eq:stochasticLp}
    \Big\|\max_{n \in \{1,\ldots,N_h\} }\Big|
    \sum_{j=1}^{n}\sum_{r=1}^m  \Gamma^j_{(r)} \Big| \Big\|_{L^p(\Omega_W)} 
    \leq C \big( 1 + \| X \|_{ 
    \mathcal{C}^{\frac{1}{2}}([0,T]; L^{2p}(\Omega_W;\R^d))}^2 \big) 
    |h|^{\min(\frac{1}{2} + \gamma, 1)}. 
  \end{equation*}
\end{lemma}

\begin{proof}
  For each fixed $r \in \{1,\ldots,m\}$ we can write
  \begin{align*}
    \Gamma^j_{(r)} = \int_{t_{j-1}}^{t_j} G^r(s) \diff{W^r(s)},
  \end{align*}
  with integrand $G^r \colon [0,T] \times \Omega_W \to \R^d$ defined by $G^r(0)
  \equiv 0$ and for each $j \in \{1,\ldots,N_h\}$ and $s \in (t_{j-1}, t_j]$ by
  \begin{align*}
    G^r(s) := g^r(s,X(s)) - g^r( t_{j-1}, X(t_{j-1}))
    - \sum_{r_2 = 1}^m g^{r,r_2}( t_{j-1}, X(t_{j-1}) ) I_{(r_2)}^{t_{j-1},s}.
  \end{align*}
  From this it follows directly that $G^r$ is predictable. The linear growth
  conditions on $g^r$ and $g^{r,r_2}$ together with Theorem~\ref{th:Lp_est} 
  also ensure the integrability of $G^r$.
  Therefore, $\Gamma^j_{(r)}$ is a well-defined stochastic integral.
  Consequently, the discrete-time process $n \mapsto \sum_{j = 1}^n
  \Gamma^j_{(r)} \in L^p(\Omega_W;\R^d)$
  is a martingale with respect to the filtration
  $(\F^W_{t_n} )_{n \in \{0,1,\ldots,N_h\}}$.
  Hence, the Burkholder-Davis-Gundy inequality (Theorem~\ref{th:discreteBDG})
  is applicable and we obtain 
  \begin{align}
    \label{eq6:step1}
    \begin{split}
      &\Big\| \max_{n \in \{1,\ldots,N_h\} } \big| \sum_{j=1}^{n}
      \sum_{r = 1}^m \Gamma^j_{(r)} \big|\Big\|_{L^p(\Omega_W)}
      \le C_p \sum_{r = 1}^m \Big\| \Big( \sum_{j = 1}^{N_h} \big|
      \Gamma^j_{(r)} \big|^2 \Big)^{\frac{1}{2}} \Big\|_{L^p(\Omega_W)}\\
      &\quad = C_p \sum_{r = 1}^m \Big\| \sum_{j = 1}^{N_h} 
      \big| \Gamma^j_{(r)} \big|^2
      \Big\|_{L^{\frac{p}{2}}(\Omega_W)}^{\frac{1}{2}} 
      \le C_p \sum_{r = 1}^m \Big( \sum_{j=1}^{N_h} \big\| \Gamma^j_{(r)}
      \big\|_{L^p(\Omega_W;\R^d)}^2 \Big)^{\frac{1}{2}}.
    \end{split}
  \end{align}
  Moreover, an application of Theorem~\ref{th:Lpstochint} yields
  \begin{align}
    \label{eq:estGammar}
    \big\| \Gamma^j_{(r)} \big\|_{L^p(\Omega_W;\R^d)}
    &\le C_p h_j^{\frac{p - 2}{2p}} \| G^r \|_{L^p([t_{j-1}, t_j] \times
    \Omega_W;\R^d)}.
  \end{align}
  Thus, it remains to give an estimate for $\| G^r \|_{L^p([t_{j-1}, t_j] \times
  \Omega_W;\R^d)}$. To this end we add and subtract several terms and obtain
  for each $j \in \{1,\ldots,N_h\}$ and $s \in (t_{j-1}, t_j]$ 
  \begin{align*}
    &G^r(s) = \big( g^r(s,X(s)) - g^r(t_{j-1}, X(s)) \big) \\
    &\quad + \Big( g^r(t_{j-1}, X(s)) - g^r(t_{j-1}, X(t_{j-1})) -
    \frac{\partial g^r}{\partial x}(t_{j-1}, X(t_{j-1})) 
    \big( X(s) - X(t_{j - 1}) \big) \Big)\\
    &\quad + \Big( \frac{\partial g^r}{\partial x}(t_{j-1}, X(t_{j-1})) 
    \big( X(s) - X(t_{j - 1}) \big) 
    - \sum_{r_2 = 1}^m g^{r,r_2}( t_{j-1}, X(t_{j-1}) ) I_{(r_2)}^{t_{j-1},s}
    \Big)\\
    &\; =: D_1^r(s) + D_2^r(s) + D_3^r(s).
  \end{align*}
  We estimate the three terms separately. The estimate for the first term
  follows at once from Assumption~\ref{as:g}. In fact, we have
  \begin{align}
    \label{eq:estD1}
    \begin{split}
      \| D^r_1 \|_{L^p([t_{j-1}, t_j] \times \Omega_W;\R^d)}
      &\le K_g \Big( 1 + \big\| \sup_{t \in [0,T]} |X(t)|
      \big\|_{L^p(\Omega_W)}  \Big) 
      h_j^{ \min(\frac{1}{2} + \gamma,1) + \frac{1}{p} }.
    \end{split}
  \end{align}
  For the estimate of the term $D_2^r$ we first apply the mean-value theorem
  and obtain
  \begin{align*}
    D^r_2(s) &= \int_0^1 \Big( \frac{\partial g^r}{\partial x}\big( t_{j-1}, 
    X(t_{j-1}) + \rho ( X(s) - X(t_{j - 1}) ) \big) 
    - \frac{\partial g^r}{\partial x}(t_{j-1}, X(t_{j-1})) \Big) \diff{\rho}\\
    &\qquad \times \big( X(s) - X(t_{j - 1}) \big).
  \end{align*}
  Then we make use of the Lipschitz continuity of $\frac{\partial
  g^r}{\partial x}$ and arrive at
  \begin{align*}
    | D^r_2(s) | \le \frac{1}{2} K_g \big| X(s) - X(t_{j-1}) \big|^2. 
  \end{align*}
  Therefore, by an application of \eqref{eq:Lp2}
  \begin{align}
    \label{eq:estD2}
    \begin{split}
      \| D^r_2 \|_{L^p([t_{j-1}, t_j] \times \Omega_W;\R^d)} 
      &\le \frac{1}{2} K_g \Big( \int_{t_{j-1}}^{t_j} 
      \E_W \big[ | X(s) - X(t_{j-1}) |^{2p} \big] \diff{s}
      \Big)^{\frac{1}{p}}\\
      &\le \frac{1}{2} K_g \| X \|_{\mathcal{C}^{\frac{1}{2}}([0,T];
      L^{2p}(\Omega_W;\R^d))}^2 h_j^{1 + \frac{1}{p}}. 
    \end{split}
  \end{align}
  For the estimate of $D^r_3$ first recall the definition of $g^{r,r_2}$ from
  \eqref{eq:diffg}. In addition, we also insert the integral equation
  \eqref{eq:exactint} for $X(s) - X(t_{j-1})$
  and obtain for all $s \in [t_{j-1}, t_j]$
  \begin{align*}
    &D^r_3(s) =  \int_{t_{j-1}}^s \frac{\partial g^r}{\partial x}
    (t_{j-1}, X(t_{j-1})) f(u,X(u)) \diff{u} \\
    &\quad +  \sum_{r_2 = 1}^m \int_{t_{j-1}}^s
    \frac{\partial g^r}{\partial x}(t_{j-1}, X(t_{j-1})) 
    \big( g^{r_2}(u,X(u)) - g^{r_2}(t_{j-1}, X(t_{j-1})) \big)
    \diff{W^{r_2}(u)},
  \end{align*}
  where we also made use of the fact that the random 
  matrix $\frac{\partial g^r}{\partial x}(t_{j-1},
  X(t_{j-1}))$ is $\F_{t_{j-1}}^W$-measurable and is therefore interchangeable
  with the stochastic integral. By the linear growth of $f$ and
  the boundedness of $\frac{\partial g^r}{\partial x}$ we then obtain the
  estimate 
  \begin{align*}
    \Big| \int_{t_{j-1}}^s \frac{\partial g^r}{\partial x}
    (t_{j-1}, X(t_{j-1})) f(u,X(u)) \diff{u} \Big|
    \le \tilde{K}_f K_g \big(1 + \sup_{t \in [0,T]} | X(t) | \big) h_j.
  \end{align*}
  Moreover, from the boundedness of $\frac{\partial
  g^r}{\partial x}$, the H\"older and Lipschitz continuity of $g^{r_2}$, and an
  application of Theorem~\ref{th:Lpstochint} we also get for all $r, r_2 \in
  \{1,\ldots,m\}$ that
  \begin{align*}
    &\Big\| \int_{t_{j-1}}^s
    \frac{\partial g^r}{\partial x}(t_{j-1}, X(t_{j-1})) 
    \big( g^{r_2}(u,X(u)) - g^{r_2}(t_{j-1}, X(t_{j-1})) \big)
    \diff{W^{r_2}(u)} \Big\|_{L^p(\Omega_W;\R^d)}\\
    &\quad \le C_p K_g^2 h_{j}^{\frac{p - 2}{2p}} \Big( \int_{t_{j-1}}^{t_j}
    \E_W \big[ \big( (1 + |X(u)|) |u - t_{j-1}|^{\min(\frac{1}{2}+\gamma,1)}
    \\&\qquad\qquad\qquad  \qquad
    + | X(u) - X(t_{j-1}) | \big)^p \big] \diff{u} \Big)^{\frac{1}{p}}\\
    &\quad \le C_p K_g^2 \big( 1 + 
    \| X \|_{\mathcal{C}^{\frac{1}{2}}([0,T];L^{p}(\Omega_W;\R^d))} \big)
    h_{j}, 
  \end{align*}
  since $\frac{1}{2} \le \min(\frac{1}{2}+\gamma,1)$ and 
  $h_{j}^{\frac{p - 2}{2p} + \frac{1}{p} + \frac{1}{2} } = h_j$. In sum, after
  integrating these estimates over $[t_{j-1},t_j]$ we obtain
  the estimate
  \begin{align}
    \label{eq:estD3}
    \begin{split}
      \| D^r_3 \|_{L^p([t_{j-1}, t_j] \times \Omega_W;\R^d)}
      \le K_g (\tilde{K}_f +  C_p K_g) \big( 1 + \| X
      \|_{\mathcal{C}^{\frac{1}{2}}([0,T];L^{p}(\Omega_W;\R^d))} \big)  
      h_j^{1 + \frac{1}{p}}.
    \end{split}
  \end{align}
  Altogether, by combining \eqref{eq:estD1}, \eqref{eq:estD2}, and
  \eqref{eq:estD3} and due to $\| X
  \|_{\mathcal{C}^{\frac{1}{2}}([0,T];L^{p}(\Omega_W;\R^d))} 
  \le \| X \|_{\mathcal{C}^{\frac{1}{2}}([0,T];L^{2p}(\Omega_W;\R^d))}$
  we finally arrive at
  \begin{align*}
    \| G^r \|_{L^p([t_{j-1}, t_j] \times \Omega_W;\R^d)}
    \le C \big( 1 + \| X
    \|_{\mathcal{C}^{\frac{1}{2}}([0,T];L^{2p}(\Omega_W;\R^d))}^2 \big)   
    h_j^{\min(\frac{1}{2} + \gamma,1) + \frac{1}{p}},
  \end{align*}
  for some constant $C$ only depending on $\tilde{K}_f$, $K_g$, $p$.
  Inserting this into \eqref{eq:estGammar} and \eqref{eq6:step1} then
  yields the assertion.
\end{proof}

\begin{theorem} 
  \label{th:Consistence}
  Let Assumptions~\ref{as:X0} to \ref{as:g} be satisfied with
  $p \in [2,\infty)$ and $\gamma \in (0,1]$.
  Then, the residual $R_h = R_h[X|_{\pi_h}]$ defined in \eqref{eq:residual}
  of the exact solution $X$ can be estimated by
  \begin{align*}
    \| R_h \|_{S,p} 
    &\le \| X_h^0 - X(0) \|_{L^p(\Omega;\R^d)} + C 
    \big( 1 + \| X \|_{
    \mathcal{C}^{\frac{1}{2}}([0,T]; L^{2p}(\Omega_W;\R^d))}^2 \big)
    |h|^{\min(\frac{1}{2} + \gamma, 1)},
  \end{align*}
  where the constant $C \in (0,\infty)$ only depends on $T$, $p$, $m$,
  $\tilde{K}_f$, and $K_g$. In particular, if $X_h^0 = X(0) = X_0$, then
  the drift-randomized Milstein method \eqref{eq:RandM} is consistent 
  of order $\beta = \min(\frac{1}{2} + \gamma, 1)$.
\end{theorem}

\begin{proof} 
  Let $\pi_h = \{ 0 = t_0 < t_1 < \ldots < t_{N_h} = T\}$ be an arbitrary
  temporal grid with maximum step size $|h| \le \min(1,T)$. 
  First recall the definition \eqref{eq:residual} of the residual of
  $X|_{\pi_h}$ for $j \in \{1,\ldots,N_h\}$
  \begin{align*}
    R_h^j := R_h^j[ X|_{\pi_h} ] = X(t_j) - X(t_{j - 1}) 
    - \Phi^j(X(t_{j-1}), \tau_j).
  \end{align*}
  We have to estimate $R_h$ with respect to the Spijker norm
  $\| \cdot \|_{S,p}$. To this end  
  we expand the residual by inserting \eqref{eq:exactint} and
  \eqref{eq:PhiM_expression}. For $j \in \{1,\ldots,N_h\}$
  we then have
  \begin{align*}
    R_h^j &= \int_{t_{j-1}}^{t_j} f(s,X(s)) \diff{s} 
    - h_j f(t_{j-1} + \tau_j h, X(t_{j-1} + \tau_j h)) \\
    &\quad + h_j \big(f(t_{j-1} + \tau_j h, X(t_{j-1} + \tau_j h)) -
    f(t_{j-1} + \tau_j h, \Psi_h^j ( X(t_{j-1}), \tau_j) ) \big) \\
    &\quad +\sum_{r = 1}^m \Gamma_{(r)}^j,  
  \end{align*}
  where $\Gamma_{(r)}^j$ is the same as in Lemma~\ref{lm:milstein1}. After
  summing over $j \in \{1,\ldots,n\}$ and taking the Euclidean norm in $\R^d$
  we get
  \begin{align*}
    \Big| \sum_{j = 1}^n R_h^j \Big| &\le \Big|\int_0^{t_n} f(s,X(s)) \diff{s}
    - Q_{\tau,h}^n[f(\cdot,X(\cdot))] \Big| \\
    &\quad + K_f \sum_{j = 1}^n
    h_j \big|  X(t_{j-1} + \tau_j h) - \Psi_h^j ( X(t_{j-1}), \tau_j) \big| 
    + \Big| \sum_{r = 1}^m \sum_{j = 1}^n \Gamma_{(r)}^j \Big|,
  \end{align*}
  where we also inserted the definition of the randomized quadrature rule
  $Q_{\tau,h}^n$ from \eqref{eq:defQtau} and made use of the Lipschitz
  continuity of $f$. Next, we take the maximum over all $n \in
  \{1,\ldots,N_h\}$ and apply the $L^p(\Omega)$-norm. This yields the
  estimate
  \begin{align*}
    &\Big\| \max_{n \in \{1,\ldots,N_h\}} \Big| \sum_{j = 1}^n R_h^j \Big| \,
    \Big\|_{L^p(\Omega)} \\
    &\quad \le \Big\| \max_{n \in \{1,\ldots,N_h\}} \Big| \int_0^{t_n} f(s,X(s))
    \diff{s} - Q_{\tau,h}^n[f(\cdot,X(\cdot))] \Big| \, \Big\|_{L^p(\Omega)}\\
    &\qquad +  K_f \sum_{j = 1}^{N_h} h_j 
    \big\|  X(t_{j-1} + \tau_j h) - \Psi_h^j ( X(t_{j-1}), \tau_j)
    \big\|_{L^p(\Omega;\R^d)}\\
    &\qquad+ \Big\| \max_{n \in \{1,\ldots,N_h\}} \Big|
    \sum_{r = 1}^m \sum_{j = 1}^n \Gamma_{(r)}^j \Big| \,
    \Big\|_{L^p(\Omega)}.
  \end{align*}
  Next, from Assumption~\ref{as:f} it follows that
  the process $Y(s) := f(s,X(s))$, $s \in [0,T]$, is H\"older continuous with
  exponent $\nu = \min(\gamma, \frac{1}{2})$. In particular,
  \begin{align*}
    \| Y \|_{\mathcal{C}^{\nu}([0,T],L^p(\Omega_W;\R^d))}
    \le \tilde{K}_f \big( 1 + \| X
    \|_{\mathcal{C}^{\frac{1}{2}}([0,T],L^p(\Omega_W;\R^d))} \big).
  \end{align*}
  Therefore, Theorem~\ref{th4:randRiemann} is applicable and yields
  \begin{align*}
    &\Big\| \max_{n \in \{1,\ldots,N_h\}} \Big| \int_0^{t_n} f(s,X(s))
    \diff{s} - Q_{\tau,h}^n[f(\cdot,X(\cdot))] \Big| \, \Big\|_{L^p(\Omega)}\\
    &\qquad \le C T^{\frac{1}{2}} \tilde{K}_f \big( 1 + \| X
    \|_{\mathcal{C}^{\frac{1}{2}}([0,T],L^p(\Omega_W;\R^d))} \big)
    |h|^{\min(\frac{1}{2} + \gamma, 1)},
  \end{align*}
  since $\nu + \frac{1}{2} = \min(\frac{1}{2} + \gamma, 1)$.

  In addition, Lemma~\ref{lm:milstein1} ensures 
  \begin{align*}
    \Big\| \max_{n\in \{1,\ldots, N_h\}}
    \Big|\sum_{j=1}^{n}\sum_{r=1}^m \Gamma^r_j\Big|\,
    \Big\|_{L^p(\Omega_W)}
    \le C \big( 1 + \| X \|_{ 
    \mathcal{C}^{\frac{1}{2}}([0,T]; L^{2p}(\Omega_W;\R^d))}^2 \big) 
    |h|^{\min(\frac{1}{2} + \gamma, 1)}.
  \end{align*}
  Therefore, it remains to give an estimate for
  \begin{align}
    \label{eq6:term3}
    \begin{split}
      &\big\|  X(t_{j-1} + \tau_j h) - \Psi_h^j ( X(t_{j-1}), \tau_j)
      \big\|_{L^p(\Omega;\R^d)}\\
      &\quad \le \Big\| \int_{t_{j-1}}^{t_{j-1}+\tau_j h_j} 
      \big( f(s,X(s))-f\left(t_{j-1},X(t_{j-1})\right) \big) \diff{s} 
      \Big\|_{L^p(\Omega;\R^d)}\\
      &\qquad + \sum_{r=1}^m \Big\| 
      \int_{t_{j-1}}^{t_{j-1}+\tau_jh_j} \big(g^r\big(s,X(s)\big) 
      - g^r\big(t_{j-1},X(t_{j-1})\big)\big) \diff{W^r(s)} 
      \Big\|_{L^p(\Omega;\R^d)},
    \end{split}
  \end{align}
  where we inserted \eqref{eq:exactint} and \eqref{eq:PsiM}. Then, 
  by an application of Assumption~\ref{as:f} to the first term on the right
  hand side in \eqref{eq6:term3} we obtain
  \begin{align*}
    &\Big\| \int_{t_{j-1}}^{t_{j-1}+\tau_j h_j} 
    \big( f(s,X(s))-f\left(t_{j-1},X(t_{j-1})\right) \big) \diff{s} 
    \Big\|_{L^p(\Omega;\R^d)}\\
    &\; \le \Big( \E_\tau \Big[ \E_W \Big[ \Big(
    \int_{t_{j-1}}^{t_{j-1} + \tau_j h_j }
    \big| f(s,X(s))-f\left(t_{j-1},X(t_{j-1})\right) \big| \diff{s} \Big)^p
    \Big] \Big] \Big)^{\frac{1}{p}}\\
    &\; \le K_f \Big( \E_\tau \Big[ \E_W \Big[ \Big(
    \int_{t_{j-1}}^{t_{j}}
    \big( (1 + |X(t_{j-1})| ) |s - t_{j-1}|^\gamma + | X(s) - X(t_{j-1}) |
    \big) \diff{s} \Big)^p \Big] \Big] \Big)^{\frac{1}{p}}\\
    &\; \le K_f \big( 1 + \| X \|_{ 
    \mathcal{C}^{\frac{1}{2}}([0,T]; L^{p}(\Omega_W;\R^d))} \big)
    h_j^{1 + \min(\gamma, \frac{1}{2})}.
  \end{align*}
  Moreover, an application of 
  Theorem~\ref{th:Lpstochint} and Assumption~\ref{as:g} to the second
  term on the right hand side of \eqref{eq6:term3} yields
  \begin{align*}
    &\Big\| \int_{t_{j-1}}^{t_{j-1}+\tau_jh_j} \big(g^r\big(s,X(s)\big) 
    - g^r\big(t_{j-1},X(t_{j-1})\big)\big) \diff{W^r(s)} 
    \Big\|_{L^p(\Omega;\R^d)}^p\\
    &\; = \E_\tau \Big[ \E_W \Big[ \Big|
    \int_{t_{j-1}}^{t_{j-1}+\tau_jh_j} \big(g^r (s,X(s)) 
    - g^r( t_{j-1},X(t_{j-1}) ) \big) \diff{W^r(s)} \Big|^p \Big] \Big]
    \\
    &\; \le C_p^p h_j^{\frac{p-2}{2}}  
    \big\| g^r ( \cdot, X(\cdot) ) - g^r( t_{j-1},X(t_{j-1}) )
    \big\|_{L^p([t_{j-1},t_j] \times \Omega_W;\R^d)}^p\\
    &\; \le K_g^p C_p^p h_j^{\frac{p-2}{2}}
    \int_{t_{j-1}}^{t_j} \E_W \Big[ \Big(
    (1 + |X(t_{j-1})| ) |s - t_{j-1}|^{\min(\frac{1}{2}+\gamma,1)}
    \\
    &\qquad \qquad \qquad \qquad + | X(s) - X(t_{j-1}) |
    \Big)^p\Big] \diff{s} \\
    &\;\le K_g^p C_p^p  \big( 1 + \| X \|_{ 
    \mathcal{C}^{\frac{1}{2}}([0,T]; L^{p}(\Omega_W;\R^d))} \big)^p
    h_j^{p},
  \end{align*}
  since $h_j^{\frac{p - 2}{2} + \frac{p}{2} + 1  } = h_j^p$. Taking the $p$-th
  root and inserting this into \eqref{eq6:term3} then yields 
  \begin{align*}
    &\big\|  X(t_{j-1} + \tau_j h) - \Psi_h^j ( X(t_{j-1}), \tau_j)
    \big\|_{L^p(\Omega;\R^d)}\\
    &\quad \le ( K_f + m K_g C_p) \big( 1 + \| X \|_{ 
    \mathcal{C}^{\frac{1}{2}}([0,T]; L^{p}(\Omega_W;\R^d))} \big)
    h_j. 
  \end{align*}
  Altogether, we have shown that
  \begin{align*}
    \| R_h \|_{S,p} 
    &= \| X_h^0 - X_0 \|_{L^p(\Omega;\R^d)}
    + \Big\| \max_{n \in \{1,\ldots,N_h\}} \Big| \sum_{j = 1}^n R_h^j \Big| \,
    \Big\|_{L^p(\Omega)} \\
    &\le \| X_h^0 - X_0 \|_{L^p(\Omega;\R^d)} + C 
    \big( 1 + \| X \|_{
    \mathcal{C}^{\frac{1}{2}}([0,T]; L^{2p}(\Omega_W;\R^d))}^2 \big)
    |h|^{\min(\frac{1}{2} + \gamma, 1)}.
  \end{align*}
  This completes the proof.
\end{proof}

The proof of Theorem~\ref{th:main2} is now a simple consequence of the above.

\begin{proof}[Proof of Theorem \ref{th:main2}]
  Since the drift-randomized Milstein method is bistable (see
  Theorem~\ref{th:Bistability}) we apply
  the bistability inequality \eqref{eq:bistability} with 
  $Y_h := X|_{\pi_h}$. Then, an application of
  Theorem~\ref{th:Consistence} yields 
  \begin{align*}
    \| X_h - X|_{\pi_h} \|_{p,\infty}
    &= \Big\|\max_{n \in \{1,\ldots,N_h\}}
    |X_h^n-X(t_n)| \Big\|_{L^p(\Omega)}\\ 
    &\le C_2 \| R_h[ X|_{\pi_h} ] \|_{S,p}\\
    &\le C \big( 1 + \| X \|_{
    \mathcal{C}^{\frac{1}{2}}([0,T]; L^{2p}(\Omega_W;\R^d))}^2 \big)
    |h|^{\min(\frac{1}{2} + \gamma, 1)},
  \end{align*}
  as claimed.
\end{proof}

%%%%%%%%%%%%%%%%%%%%%%%%%%%%%%%
% NUMERICAL EXAMPLES
%%%%%%%%%%%%%%%%%%%%%%%%%%%%%%%

\section{Implementation and a numerical example}
\label{sec:examples}

In this section the implementation of the randomized Milstein method is
discussed and a numerical experiment is conducted. 

%\subsection{How to implement the randomized Milstein method}
% Implementation

Being an explicit method, the implementation of the drift-randomized Milstein
method is mostly straightforward. The only obstacle that needs to
be treated carefully is the simulation of the intermediate stochastic
increments $I_{(r)}^{t_{j-1},t_{j-1}+ \tau_j h_j }$ for all $r \in
\{1,\ldots,m\}$ in the computation of $X_h^{j,\tau}$ in \eqref{eq:RandM}. In
particular, it is important that 
the additional information on the path of the Wiener process at
the (random) intermediate 
time point $t_{j-1}+ \tau_j h_j$ is also taken into account
in the computation of $I_{(r)}^{t_{j-1},t_j}$ and $I_{(r_1,r_2)}^{t_{j-1},
t_j}$. This is ensured by the following step by step procedure:
% for the simulation of $X_h^j$ given $X_{h}^{j-1}$:

\begin{enumerate}
  \item First simulate $\tau_j \sim \mathcal{U}(0,1)$ and set $\theta_j :=
    t_{j-1}+ \tau_j h_j$.
  \item Then simulate $I_{(r)}^{t_{j-1},\theta_j}$ and
    $I_{(r_1,r_2)}^{t_{j-1},\theta_j}$ jointly for all $r,r_1,r_2 \in
    \{1,\ldots,m\}$ as in the case of the classical 
    Milstein method, see for instance \cite[Sec.~5.8]{kloeden1999}. 
  \item In the same way simulate $I_{(r)}^{\theta_j, t_j}$ and
    $I_{(r_1,r_2)}^{\theta_j, t_j}$ for all $r,r_1,r_2 \in \{1,\ldots,m\}$.
  \item Then we obtain
    $I_{(r)}^{t_{j-1},t_j}$ and $I_{(r_1,r_2)}^{t_{j-1}, t_j}$ from
    \begin{align*}
      I_{(r)}^{t_{j-1},t_j} = I_{(r)}^{t_{j-1},\theta_j} + I_{(r)}^{\theta_j,
      t_j} 
    \end{align*}
    as well as (Chen's relation)
    \begin{align*}
      I_{(r_1,r_2)}^{t_{j-1}, t_j} = I_{(r_1,r_2)}^{t_{j-1},\theta_j} + 
      I_{(r_1,r_2)}^{\theta_j, t_j} + I_{(r_1)}^{t_{j-1},\theta_j}
      I_{(r_2)}^{\theta_j, t_j}. 
    \end{align*}
  \item Compute $X_h^j$ as defined in \eqref{eq:RandM}.
\end{enumerate}

Listing~\ref{py:RandM} shows an implementation  of method \eqref{eq:RandM} in
the case of a 1-dimensional Wiener process ($m=1$) in \textsc{Python}.
This allows us to
compute the iterated stochastic increment $I_{(1,1)}^{s,t}$ for $s,t \in
[0,T]$, $s<t$, efficiently by the relationship
\begin{align*}
  I_{(1,1)}^{s,t} = \frac{1}{2}\big( ( I_{(1)}^{s,t} )^2 - (t-s) \big).
\end{align*}
This algorithm is easily adapted to the case of
multi-dimensional Wiener processes if the coefficient functions $g^{r_1,r_2}$
defined in \eqref{eq:diffg} satisfy the commutativity condition
$g^{r_1,r_2} = g^{r_2, r_1}$ for all $r_1, r_2 \in \{1,\ldots,m\}$. Compare
further with \cite[Sec.~10.3]{kloeden1999}.

\begin{lstlisting}[caption={A sample implementation of \eqref{eq:RandM} in
\textsc{Python}}, 
language=Python, label=py:RandM]
import numpy as np

def f(t,x):
    return [...]

def g(t,x):
    return [...]

def Dg_g(t,x):
    return [...]

def RandMilstein(pi_h,X0):
    # input:    temporal grid pi_h, initial value X0
    # output:   one trajectory of the rand. Milstein method
    
    d = np.array(X0).size
    h = np.diff(pi_h)   # vector of step sizes
    N_h = h.size
    X_h = np.zeros( (N_h+1, d) )  # allocating X_h
    X_h[0,:] = np.array(X0)   # initial condition

    for j in xrange(N_h):
        # step (1):
        tau_j = np.random.rand() 
        theta_j = pi_h[j] + tau_j*h[j]
        # step (2):
        I_1 = np.sqrt(tau_j*h[j])*np.random.normal()   
        I_11 = ( I_1**2 - tau_j*h[j])/2.        
        # step (3):
        J_1 = np.sqrt((1-tau_j)*h[j])*np.random.normal()   
        J_11 = ( J_1**2 - (1-tau_j)*h[j])/2.
        # step (4):
        K_1 = I_1 + J_1
        K_11 = I_11 + J_11 + I_1*J_1
        # step (5):
        X_tau = X_h[j,:] + tau_j*h[j]*f(pi_h[j], X_h[j,:]) \
                + g(pi_h[j], X_h[j,:])*I_1
        X_h[j+1,:] = X_h[j,:] + h[j]*f(theta_j, X_tau) \
                + g(pi_h[j], X_h[j,:])*K_1 \
                + Dg_g(pi_h[j], X_h[j,:])*K_11
    return X_h
\end{lstlisting}

%\subsection{SDE with Non-differentiable drift}
% Numerical Examples
Next, we consider the numerical solution of the scalar SDE 
\begin{align}
  \label{eq:SDEexample1}
  \begin{split}
    \begin{cases}
      &\diff{X(t)} = (\mu |X(t)|+|\sin( w_1 t)|)\diff{t}+ |\cos( w_2
      t)|X(t)\diff{W(t)}, \quad t \in [0,T],\\
      &\ \ X(0) = X_0,
    \end{cases}
  \end{split}
\end{align}
where $\mu$, $w_1$ and $w_2$ are real constants. It is easily verified that
Assumptions~\ref{as:f} and \ref{as:g} are fulfilled. In the
experiment, we set $\mu=-0.01$, $w_1=2^{6}\pi$, $w_2=1,$ $X_0=1.1$ and $T=1$.
We compare the numerical solution of \eqref{eq:SDEexample1} by 
the drift-randomized Milstein scheme \eqref{eq:RandM} and its 
classical counter-part. We approximate the error only at the terminal 
time $T=1$ with respect to the $L^2$-norm 
by a Monte Carlo simulation with $1000$ independent samples.
Hereby, the reference solution is obtained using the randomized Milstein
scheme with a finer step size of $h_{\mathrm{ref}} = 2^{-15} T$. 

In Figure~\ref{fig_error}, we plot the 
root-mean-squared errors against the underlying step size, i.e., the number $n$
on the $x$-axis indicates the corresponding simulation is based on the step
size $h = 2^{-n} T $. The finest step size here is $2^{-14}T$. The two sets of
error data are fitted with a linear function via linear regression
respectively, where the slope of the line indicates the average order of
convergence.  It is noted that the classical Milstein scheme does not begin to
converge until $n=6$. The reason for this is, that
for any coarser (equidistant) step size larger than $2^{-6}T$ the classical Milstein scheme cannot
distinguish the term $|\sin( w_1 t)|$ in the drift from the zero function.
In contrast, the randomized Milstein method shows better results already for
much coarser step sizes. The experimental order of convergence is $0.83$ up to $n=6$ compared with the order $0.19$ via classical Milstein. Note that afterwards the error from classical method begin to shrink at a faster pace and eventually decay at the same rate as randomized Milstein method.  

Finally, we briefly compare the computational efficiency of the two methods.
Clearly, due to the additional computation of $X_h^{j,\tau}$ the randomized
Milstein method is \eqref{eq:RandM} approximately twice as expensive as the
classical one with the same step size. We also observe this in our
experiment, since the data points of the classical Milstein method are shifted to
the left in Figure~\ref{fig_time}, where the CPU times of these schemes are
plotted versus their accuracy. But due to its better accuracy 
the randomized Milstein method is superior for all the step
sizes larger than $2^{-6}T$. However, when even smaller step sizes are
considered, the error of the classical Milstein method will quickly decrease to
the level of the randomized one. In the scalar case the  

\begin{figure}%
  \begin{center}%
      \caption{\small Numerical experiment for  SDE \eqref{eq:SDEexample1}:
      Step sizes versus $L^2$ errors
      \label{fig_error}} 
      \includegraphics[width=0.7\textwidth]{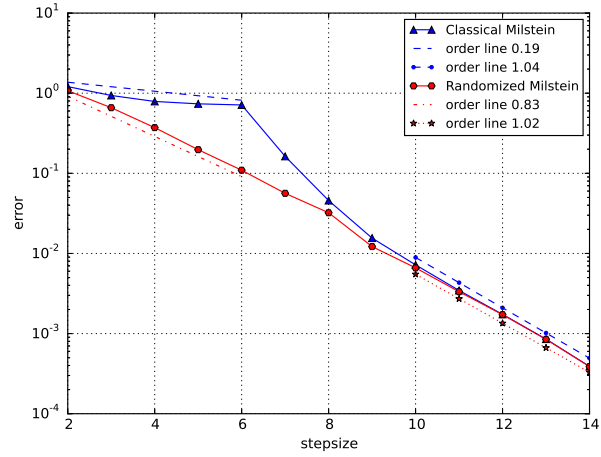}
      \\[-10pt] %
  \end{center}
\end{figure}   

\begin{figure}%
  \begin{center}%
      \caption{\small Numerical experiment for  SDE \eqref{eq:SDEexample1}:
      CPU time versus $L^2$ errors \label{fig_time}} 
      \includegraphics[width=0.7\textwidth]{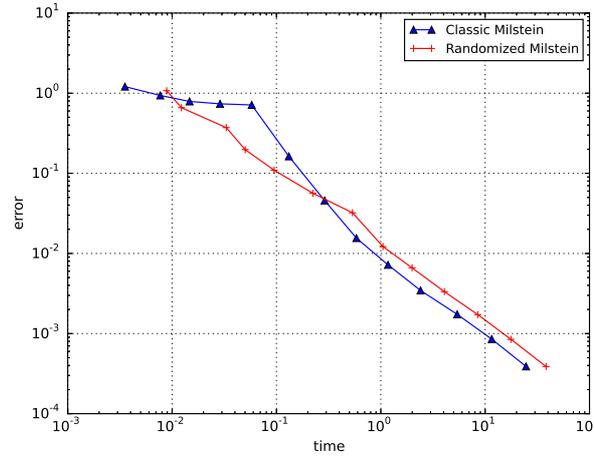}
      \\[-10pt] %
  \end{center}
\end{figure}

\section*{Acknowledgement}

This research was carried out in the framework of \textsc{Matheon}
supported by Einstein Foundation Berlin. The authors also gratefully
acknowledge financial support by the German Research Foundation through the
research unit FOR 2402 -- Rough paths, stochastic partial differential
equations and related topics -- at TU Berlin.

%\bibliographystyle{plain}
%\bibliography{../literature/lit}

\end{document}